\documentclass[10pt]{article}
   \usepackage{graphicx}
   \usepackage{epsfig}
   \usepackage{amssymb}
   \usepackage{amsmath}
   \usepackage{amsthm}
   \newtheorem{definition}{Definition}[section]
   \newtheorem{lemma}{Lemma}[section]
   \newtheorem{theorem}{Theorem}[section]
   \newtheorem{proposition}{Proposition}[section]
   \newtheorem{corollary}{Corollary}[section]
   \newtheorem{remark}{Remark}[section]

   {\end{description}}

   {\hfill $\bullet$ \\}

   \newcommand{\be}{\begin{equation}}
   \newcommand{\ee}{\end{equation}}

\begin{document}
    \title{Spectral Distribution in the Eigenvalues Sequence of Products of g-Toeplitz Structures}
   \author{Eric Ngondiep \thanks{Department of Mathematics and Statistics, College of Science, Al-Imam Muhammad Ibn Saud
   Islamic University (IMIU), 90950 Riyadh,Kingdom of Saudi Arabia.
       Email: ericngondiep@gmail.com, engondiep@imamu.edu.sa}}
    \date{\text{\,\,}}
   \maketitle

    \textbf{Abstract.}
    Starting from the definition of an $n\times n$ $g$-Toeplitz matrix, $T_{n,g}(u)=\left[\widehat{u}_{r-gs}\right]_{r,s=0}^{n-1},$
    where $g$ is a given nonnegative parameter, $\{\widehat{u}_{k}\}$ is the sequence of Fourier coefficients of the Lebesgue
   integrable function $u$ defined over the domain $\mathbb{T}=(-\pi,\pi]$, we consider the product of $g$-Toeplitz
   sequences of matrices, $\{T_{n,g}(f_{1})T_{n,g}(f_{2})\},$ which extends the product of Toeplitz structures,
    $\{T_{n}(f_{1})T_{n}(f_{2})\},$ in the case where the symbols $f_{1},f_{2}\in L^{\infty}(\mathbb{T}).$ Under suitable assumptions,
    the spectral distribution in the eigenvalues sequence is completely characterized for the products of $g$-Toeplitz structures.
   Specifically, for $g\geq2$ our result shows that the sequences $\{T_{n,g}(f_{1})T_{n,g}(f_{2})\}$ are clustered to zero. This extends
   the well-known result, which concerns the classical case (that is, $g=1$) of products of Toeplitz matrices. Finally, a large
   set of numerical examples confirming the theoretic analysis is presented and discussed.\\

   \\
   {\bf AMS SC: 65F10, 15A18, 47B36, 47B65}.

   \section{Introduction}\label{I}
   Let $f$ be a Lebesgue function defined on the interval $(-\pi,\pi]$. We recall that for a given nonnegative
   integer $g,$ an $n\times n$ matrix $A_{n,g},$ is called $g$-Toeplitz if $A_{n,g}=\left[\hat{f}_{r-gs}\right]_{r,s=0}^{n-1}$.
   In this case, a $g$-Toeplitz matrix is denoted by $T_{n,g}(f)$ and the sequence $\{\hat{f}_{k}\}_{k}$ of entries
   of $T_{n,g}(f),$ can be interpreted as the sequence of Fourier coefficients of an integrable function $f$ defined on
   $\mathbb{T}.$ In this work we are motivated by the variety of fields where such matrices can be encountered such as,
   e.g., multigrid methods \cite{13nss}, wavelet analysis together with the subdivision algorithms, or equivalently, in the
   associated refinement equations, see \cite{5nss, 7nss} and the references therein. Furthermore, interesting connections
   between dilation equations in the wavelets context and multigrid algorithms \cite{13nss, 27nss} were proven by Gilbert Strang
   \cite{23nss} when establishing the restriction/prolongation operators \cite{10nss,1nss} with boundary conditions.
   The use of different boundary conditions is quite natural when treating with signal/image restoration problems or
   differential equations, see \cite{19nss,17nss}.\\

   We denote the usual Hilbert space of square-integrable functions over the circle $G=\{z\in\mathbb{C},\text{\,}|z|=1\},$ by
   $L^{2}(G),$ and let $\mathcal{H}^{2}$ be the Hardy space of functions belonging to $L^{2}(G),$ and whose the negative
   Fourier coefficients are equal to zero. Obviously, the subset $G$ is isomorphic to the set $\mathbb{T},$ and the notation
   $G\cong \mathbb{T}$ means that both domains $G$ and $\mathbb{T}$ are isomorph. In the rest of this paper we sometimes use the
    domain $G$ or $\mathbb{T},$ depending on the context. Let us define the $g$-Toeplitz operator, $T_{f,g},$ with generating function
    $f,$ as the operator
   \begin{eqnarray*}
   % \nonumber to remove numbering (before each equation)
     T_{f,g}: \mathcal{H}^{2}&\rightarrow& \mathcal{H}^{2} \\
     u &\mapsto& P_{g}(fu),
   \end{eqnarray*}
   where $P_{g}$ is the mapping from $L^{2}(G)$ onto $\mathcal{H}^{2},$ defined as
   \begin{equation}\label{01}
    P_{g}(fu):=P^{\perp}(fu_{g}),
   \end{equation}
   where $u_{g}\in\mathcal{H}^{2},$ completely depends on the parameter $g$ and the function $u.$ For example, if $u$
   is defined on $\mathbb{T}$ by $u(t)=\exp(it),$ then the function $u_{g}$ is given by
   \begin{equation}\label{02}
    u_{g}(t)=\exp(igt),\text{\,\,\,}\forall t\in\widehat{\mathbb{T}}=\left(-\frac{\pi}{g},\frac{\pi}{g}\right].
   \end{equation}
   More specifically, $u_{g}=u\circ h_{g},$ where $h_{g}$ is the map from $\widehat{\mathbb{T}}$ onto $\mathbb{T}$ defined as
    $h_{g}(t)=gt.$ Furthermore, $P^{\perp}$ is the orthogonal
   projection from $L^{2}(G)$ onto $\mathcal{H}^{2}.$ It is worth noticing that such an operator, $T_{f,g},$ is bounded if and only
   if the symbol $f$ is in the space of (essentially) bounded functions on the circle, and its infinite matrix, $T_{g}(f),$ in the
   canonical orthonormal basis $\mathcal{B}=\{1,z,z^{2},\ldots\},$ is not (in general) constant along the diagonals, whenever $g>1.$ More
   specifically, the entries of $T_{g}(f),$ obey the rule $T_{g}(f)=\left[\hat{f}_{r-gs}\right]_{r,s=1}^{\infty},$ where the
   entries $\hat{f}_{k}$ can be interpreted as the Fourier coefficients of the symbol $f$ defined by
   \begin{equation}\label{1}
    \hat{f}_{k}=\frac{1}{2\pi}\int_{-\pi}^{\pi}f(e^{it})\exp(-ikt)dt.
   \end{equation}

   Now, let $u\in L^{1}(\mathbb{T})$ and let $n$ be a non-negative integer. By $T_{n,g}(u)$ we denote the $n\times n$ matrix
   $\left[\hat{u}_{r-gs}\right]_{r,s=1}^{n}.$ It is not hard to prove that the sequence of operators on
   $\mathcal{H}^{2},$ associated with the sequences $\{T_{n,g}(u)\}_{n=1}^{\infty},$ is an approximating sequence for the
   $g$-Toeplitz operator $T_{u,g},$ when $u\in L^{\infty}(G)$ (the space of (essentially) bounded functions on the circle), hence
    $\{T_{n,g}(u)\}_{n=1}^{\infty},$ is called a $g$-Toeplitz
    sequence. It is interesting to ask how the spectrum $\Lambda_{n,g}=\{\lambda_{1},\lambda_{2},...,\lambda_{n}\},$ of $T_{n,g}(u)$
    is associated with the set of the eigenvalues of $T_{g}(u)$ if $u\in L^{\infty}(G),$ or even if $u\in L^{1}(G),$ to analyze the
    "spectral behavior" of the sequence of sets $\{\Lambda_{n,g}\}_{n=1}^{\infty}$ (or that of the sequence
    $\{\Gamma_{n,g}\}_{n=1}^{\infty},$ where $\Gamma_{n,g}$ represents the set of singular values of $T_{n,g}(u)$). When $g=1,$
    $T_{n,1}(u)$ is nothing but the classical Toeplitz matrix $T_{n}(u),$ so an important result concerning the sequences of
    spectra is the famous Szeg\"{o} theorem which says that, if $u\in L^{\infty}(G),$ is real-valued, then
   \begin{equation}\label{2}
    \underset{n\rightarrow\infty}{\lim}\frac{1}{n}\underset{\lambda\in\Lambda_{n,1}}{\Sigma}F(\lambda)=\frac{1}{mes(G)}
    \int_{G}F(u(x))dx,
   \end{equation}
   for every continuous function $F$ with compact support (see, for example, \cite{14nss}). Furthermore, Tilli and
   Tyrtyshnikov/Zamarashkin, independently, showed that relation $(\ref{2})$ holds for any integrable function $u$
   which is just real-valued, see \cite{35sss,38sss}. Parter is the first researcher who has obtained the corresponding
   result for a complex-valued function $u$ and the sequence of sets of its singular values when replacing $u$ by
   $|u|$ under the hypothesis of continuous times uni-modular symbols, see \cite{17en}, Avram (essentially bounded symbols
   \cite{1sss}), and Tyrtyshnikov/Zamarashkin \cite{35sss,38sss}, independently, when the symbol $f$ is just integrable. The
   book \cite{8sss} gives a synopsis of all these results in chap. $5$ and $6$ and other interesting facts in chap. $3$ concerning
   the relation between the pseudospectrum of $\{T_{n}(u)\}_{n=1}^{\infty},$ and that of $T(u).$ Relation $(\ref{2})$ was
   established for a more general class of test functions $F$ in \cite{38sss,5nss} and the case of several variables
   (multilevel case) and matrix-valued functions was studied in \cite{35sss,18en} in the context
   of preconditioning (other related results were established by Linnik, Widom, Doktorski, see Section $6.9$ in \cite{8sss}).\\

   However, an obvious example where the eigenvalue result does not hold is given by the g-Toeplitz sequence related to the
   function $u(t)=\exp(-it),$ where $i^{2}=-1,$ which has only zero eigenvalues so that the requirement $(\ref{2})$
   means that $F(0)=\frac{1}{2\pi}\int_{-\pi}^{\pi}F(\exp(it))dt,$ which is far from being satisfied for all continuous
   functions with compact support, even though condition $(\ref{2})$ holds for harmonic functions (in cases like this one it
   is better to consider the pseudospectrum, see \cite{8sss}). Indeed, some authors like Tilli were able to show that, if $u$ is
   any complex-valued integrable function, then the restriction $(\ref{2})$ holds for all harmonic test functions $F$ \cite{33sss}
   and that it is even satisfied by all continuous functions with compact support as long as the generating function $u$ satisfies
   a certain geometric limitation. Moreover, the symbol $u$ must be essentially bounded and such that its (essential) range does
   not disconnect the complex plane and has empty interior, see \cite{36sss}. This set of functions is called the Tilli class.
   In other contexts, such a property is informally called "thin spectrum". We recall that the space of all essentially bounded and
   real-valued functions is obviously a subset of the Tilli class.\\

   In some recent works \cite{en,en1,enss,n1ss} we addressed the problem of asymptotic distribution of eigenvalues of $g$-Toeplitz sequences
   together with the spectral analysis of $g$-circulant matrices, in the case where the entries, $\hat{u}_{k},$ are the Fourier coefficients
   of a real-valued function $u\in L^{\infty}(\mathbb{T})$. The generalization of this analysis to the block, multilevel
   case, amounting to choose the symbol, $u,$ multivariate, i.e., defined on the set $\mathbb{T}^{d},$ and matrix-valued, i.e.,
   such that, $u(x)$ is a matrix of given dimension $p\times q$ was considered. The aim of this work is the study of the numerical
   solution of linear systems with associated $g$-Toeplitz matrices. Specifically, we focus on the problem of asymptotic analysis
   of the distribution results in the eigenvalues sequence of products of $g$-Toeplitz structures. This can be viewed as a
   preconditioning problem with $g$-Toeplitz preconditioners. When $g\geq2,$ the interesting result is that the matrix sequence
   $\{T_{n,g}(f_{1})T_{n,g}(f_{2})\}_{n}$ is clustered at zero so that the case of $g=1,$ widely studied in the literature is
   exceptional \cite{7dgmss}. However, it is worth noting to recall that the product of $g$-Toeplitz
   operators is not necessary a
   $g$-Toeplitz operator. In fact, if $g=1,$ the authors \cite{9sss,19sss} proved that the product of Toeplitz operators
   is rarely equal to a Toeplitz operator, but it turns out that the sequence of eigenvalues (singular values) of the product
   of two Toeplitz sequences is often related to the product of the two symbols in a Szeg\"{o}-type way. For the singular values
   the result is known as long as all the involved symbols are essentially bounded and, in fact, for any linear combination of
   products of Toeplitz operators, the distribution function is exactly the linear combination of the products of the generating
   functions of the sequences: the latter goes back to the work of Roch and Silbermann (see Sections $4.6$ and $5.7$ in \cite{8sss}).
   In \cite{27sss,28sss} the authors have extended the previous results by considering integrable symbols, not necessarily
   bounded, and "pseudo" inversion and the related algebra of sequences. Of course, for the eigenvalues much less is known, and one
   simple reason is that much less is true, as another basic example discussed at the beginning of Section $2$ in \cite{29sss} shows.
   The authors proved in \cite{16sss,28sss} that the eigenvalues of a non-Hermitian complex perturbation of a Jacobi matrix
   sequence, which are not necessarily real, are still distributed as the real-valued function $2\cos(t)$ over $(0,\pi),$ which
   characterizes the non-perturbed case where the Jacobi sequence is of course real and symmetric. The authors used these results
   to analyze the eigenvalue distribution of products of Toeplitz sequences, discussed, applied and extended more general tools
   introduced by Tilli \cite{36sss} and based on the Mergelyan theorem \cite{22sss}, while in
   \cite{33dgmss,20dgmss,7dgmss,13dgmss,16dgmss} the authors consider the
   spectral behavior of preconditioned non-Hermitian unilevel/multilevel block Toeplitz structures, with in certain cases, the
   symbol of the preconditioner chosen in a trigonometric polynomial so that this preconditioner is bounded and the related linear
   systems are easily solvable. In this note, the attention is focused on the product of non-Hermitian $g$-Toeplitz matrices with
   (essentially) bounded symbols whose product is real-valued. Specifically, we are interested in the following items:

   \begin{description}
     \item[(1)] Localization results for all the eigenvalues of $T_{n,g}(f_{1})T_{n,g}(f_{2}).$ When $g=1,$ the results is
     given in \cite{sss}, case where $f_{1},f_{2}\in L^{\infty}(\mathbb{T}^{d}),$ with $f_{1}f_{2}$ to be real-valued;
     \item[(2)] Spectral distribution in the eigenvalues sequence of products of $g$-Toeplitz structures
     $\{T_{n,g}(f_{1})T_{n,g}(f_{2})\}_{n\in\mathbb{N}},$ where $f_{1},f_{2}\in L^{\infty}(G):$ this is our original contribution
     and it represents an extension of the work presented in \cite{sss}, Section $3$, Theorems $3.1$-$3.6$ (see \cite{ss} for
     previous results);
     \item[(3)] A wide set of numerical examples concerning the eigenvalue distributions and the clustering properties of the
     sequence of matrices $\{T_{n,g}(f_{1})T_{n,g}(f_{2})\},$ are considered and discussed.
   \end{description}

    The paper is organized as follows. Section $\ref{II}$ is reserved for definitions and main tools. In section $\ref{III}$ we analyze
    the problem of the spectral distribution in the eigenvalues sequence of products of $g$-Toeplitz structures. A generalization
    of the distribution results to the case of blocks and multilevel setting ($e$ being the vector of all ones \cite{ss}) amounting
    to choose the multi-variate symbols is presented in section $\ref{IV}$. Section  $\ref{V}$ deals with some numerical experiments,
    while we draw in section $\ref{VI}$ the general conclusion and present the future direction of work.

   \section{Definitions and main tools}\label{II}

      We begin with some basic notations and formal definitions. For any $n\times n$ matrix, $X_{n},$ with eigenvalues
   $\lambda_{j}(X_{n})$ (respectively, singular values, $\sigma_{j}(X_{n})$), $j=1,2,...,n,$ and $p\in[1,\infty],$ we define,
   $\|X_{n}\|_{p},$ the Schatten $p$-norm of $X_{n}$ to be the $l^{p}$-norm of the vector of singular values, that is,
   \begin{equation*}
    \|X_{n}\|_{p}=\left[\overset{n}{\underset{j=1}\sum}(\sigma_{j}(X_{n}))^{p}\right]^{\frac{1}{p}}.
   \end{equation*}

   In this report, we consider the trace norm $\|\cdot\|_{1}$ together with the norm $\|\cdot\|_{\infty}$ which is known as the
   spectral norm $\|\cdot\|.$ More specifically, $\|\cdot\|$ is defined as
   \begin{equation*}
    \|X_{n}\|=\underset{x\in\mathbb{C}^{n},\text{\,}\|x\|_{2}=1}{\sup}\|X_{n}x\|.
   \end{equation*}
    We put $\Lambda_{n}=\{\lambda_{j}(X_{n}),\text{\,\,} j=1,...,n\},$  the spectrum of the matrix $X_{n}.$
   So, for any function $F$ defined on $\mathbb{C}$, the symbol $\Sigma_{\lambda}(F,X_{n}),$ stands for the mean
   \begin{equation}\label{3}
    \Sigma_{\lambda}(F,X_{n}):=\frac{1}{n}\overset{n}{\underset{j=1}\sum}F\left(\lambda_{j}(X_{n})\right)=\frac{1}{n}
    \underset{\lambda\in\Lambda_{n}}{\sum}F(\lambda),
   \end{equation}
   and the symbol $\Sigma_{\sigma}(F,X_{n}),$ denotes the corresponding expression with the singular values obtained by replacing
   the eigenvalues. Our analysis consists in explicit formulae of the distribution results of products of
   $g$-Toeplitz sequences. Following what is known in the classical case of $g=1$ (or $g=e,$ in the multilevel setting), we
   need to link the coefficients of the product of $g$-Toeplitz sequence to an (essentially) bounded function defined over
   the domain $G.$ For the general definition of Toeplitz or $g$-Toeplitz sequences (where $g$ is a $d$-dimensional
   vector of nonnegative integers), we refer the readers to \cite{nss,ns}.\\

   Now, let us introduce some important localization results taken from \cite{32sss,16dgmss}. First we recall the notion of
   essential range, $\mathcal{ER}(f),$ of a (matrix-valued) function $f.$ In the following, for every $X\subset\mathbb{C},$ $d(X,z)$
   is the (Euclidean) distance of $X,$ from the point $z\in\mathbb{C},$ and $\|\cdot\|$ denotes the spectral (Euclidean) norm
   of both vectors and matrices. For $z\in\mathbb{C},$ and $\epsilon>0,$ the disc in the complex field centered in $z,$ with
   radius $\epsilon,$ is denoted by $D(z,\epsilon).$

      \begin{definition}\label{4d}
   Given a measurable complex-valued function $h:G\rightarrow\mathbb{C},$ defined on some Lebesgue measurable set
   $G\subset\mathbb{R}^{k},$ the essential range of $h,$ denoted $\mathcal{ER}(h),$ is defined by,
   \begin{equation*}
    \mathcal{ER}(h)=\{z\in\mathbb{C}:\forall\epsilon>0,\text{\,}mes\{t\in G:h(t)\in D(z,\epsilon)\}>0\},
   \end{equation*}
   where, $mes(\cdot),$ is the Lebesgue measure in $\mathbb{R}^{k}.$
   \end{definition}

   \begin{corollary}\label{c0}
    It is easy to see that $\mathcal{ER}(h)$ is always closed (indeed: its complement is open). The function $h$ is essentially bounded
    if its essential range is bounded.
   Furthermore, if $h$ is real-valued, then the essential supremum (infimum) is defined as the supremum (infimum) of its essential
   range. In addition, if the function $h$ is an $N\times N$ matrix-valued and measurable, then the essential range of $h$ is the
   union of the essential ranges of the complex-valued eigenvalues $\lambda_{j(h)}$, $j = 1,...,N.$ Finally, it can be proven
   that $h(t)\in\mathcal{ER}(h),$ for almost every $t\in G.$
   \end{corollary}

   Now, we turn to the definition of spectral distribution and clustering, in the sense of eigenvalues and singular values,
   of a sequence of matrices (matrix-sequence), and we define the area of $K,$ in the case where, $K,$ is a compact subset
   of $\mathbb{C}.$ This definition is motivated by the Szeg$\ddot{o}$ and Tilli theorems characterizing the spectral approximation
   of a Toeplitz operator (in certain cases) by the spectra of the elements of the natural approximating matrix sequences
   $\{A_{n}\}$, where $A_{n}$ is formed by the first $n$ rows and columns of the matrix representation of the operator.
   \begin{definition}\label{5d}
   Let $\mathcal{C}_{0}(\mathbb{C})$ be the set of continuous function with bounded support defined over the complex field,
   $d$ a positive integer and $\theta$ a complex-valued measurable function defined on a set $G^{d}\subset\mathbb{C}^{d},$ of
   finite and positive Lebesgue measure $mes(G^{d})$. Here $G$ will be assumed equal to $\mathbb{T}$ (in fact, $G\cong\mathbb{T}$).
   A matrix sequence $\{A_{n}\}$ is said to be distributed (in the sense of eigenvalues) as the pair $(\theta,G^{d})$, or to have
   the distribution function $\theta,$ if for all $F\in\mathcal{C}_{0}(\mathbb{C})$, the following limit relation holds
   \begin{equation}\label{5}
    \underset{n\rightarrow\infty}{\lim}{\sum}_{\lambda}(F,A_{n})=\frac{1}{mes(G^{d})}\int_{G^{d}}F\left(\theta(t)\right)dt,
   \end{equation}
   where $\sum_{\lambda}(F,A_{n}),$ is given by relation $(\ref{3}).$ Whenever $(\ref{5})$ holds,
   $\forall F\in\mathcal{C}_{0}(\mathbb{C}),$ we write $\{A_{n}\}\sim_{\lambda}(\theta,G^{d})$.\\

   If equality $(\ref{5})$ holds for every $F\in\mathcal{C}_{0}(\mathbb{R}^{+}_{0}),$ in place of
   $F\in\mathcal{C}_{0}(\mathbb{C})$, with the singular values $\sigma_{j}(A_{n})$, $j=1,...,n,$ in place of the eigenvalues,
   and with $|\theta(t)|$ in place of $\theta(t)$, we say that the matrix sequence $\{A_{n}\}$ is distributed (in the sense
   of singular values) as the pair $(\theta,G^{d}),$ and we denote, $\{A_{n}\}\sim_{\sigma}(\theta,G^{d}),$ more specifically,
  for every $F\in\mathcal{C}_{0}(\mathbb{R}^{+}_{0}),$ we have
   \begin{equation}\label{6}
    \underset{n\rightarrow\infty}{\lim}{\sum}_{\sigma}(F,A_{n})=\frac{1}{m(G^{d})}\int_{G^{d}}F\left(|\theta(t)|\right)dt,
   \end{equation}
    where ${\sum}_{\sigma}(F,A_{n}),$ designates the corresponding expression with the singular values replacing the
    eigenvalues in $(\ref{3}).$ Furthermore, in order to treat block Toeplitz matrices, we consider measurable functions
   $\theta:\text{\,}G^{d}\rightarrow\mathcal{M}_{N}\equiv\mathcal{M}_{NN}$, where $\mathcal{M}_{MN},$ is the space of $M\times N$
   matrices with complex entries and a function is considered to be measurable if and only if the component functions are.
   In that case, $\{A_{n}\}\sim_{\lambda}(\theta,G^{d}),$ means that $M=N,$ and
   \begin{equation}\label{7}
   \underset{n\rightarrow\infty}{\lim}{\sum}_{\lambda}(F,A_{n})=\frac{1}{m(G^{d})}\int_{G^{d}}\frac{\sum_{j=1}^{N}
   F(\lambda_{j}(\theta(t)))}{N}dt,
   \end{equation}
   $\forall F\in\mathcal{C}_{0}(\mathbb{C})$, where $\lambda_{j}(\theta(t))$ in relation $(\ref{7})$ are the eigenvalues
   of the matrix $\theta(t)$.\\

   When considering $\theta$ taking values in $\mathcal{M}_{NM},$ we say that, $\{A_{n}\}\sim_{\sigma}(\theta,G^{d}),$
   when for every $F\in\mathcal{C}_{0}(\mathbb{R}^{+}_{0}),$ we have
   \begin{equation}\label{8}
    \underset{n\rightarrow\infty}{\lim}{\sum}_{\sigma}(F,A_{n})=\frac{1}{m(G^{d})}\int_{G^{d}}\frac{\sum_{j=1}^{\min\{N,M\}}
   F(\lambda_{j}(\sqrt{\theta^{*}(t)\theta(t)}))}{\min\{N,M\}}dt.
   \end{equation}

   Finally, we say that two matrix sequences $\{X_{n}\}$ and $\{Y_{n}\},$ are equally distributed in the sense of eigenvalues
   $\lambda$ (or singular values $\sigma$) if $\forall F\in\mathcal{C}_{0}(\mathbb{C}),$ we have
   \begin{equation}\label{9}
    \underset{n\rightarrow\infty}{\lim}\left[{\sum}_{\nu}(F,X_{n})-{\sum}_{\nu}(F,Y_{n})\right]=0
   \end{equation}
   with $\nu=\lambda$ ($\nu=\sigma$), respectively.
   \end{definition}

   Noting that two matrix sequences having the same distribution function are equally distributed. On the other hand,
   two equally distributed matrix sequences may be not associated with a distribution function at all. To describe
   what the distribution result (in the sense of eigenvalues) really means about the asymptotic qualities of the spectrum,
   we will introduce more concrete characterizations of sequences, $\{\Lambda_{n}\},$ such as "clustering", where as above,
   $\Lambda_{n},$ is the set of eigenvalues of $A_{n}$.
   \begin{definition}\label{6d}
   Let $\{A_{n}\}$ be the sequence of matrices with $A_{n}$ of order $n$ and let $S\subset\mathbb{C}$ be a closed subset
   of $\mathbb{C}.$ We say that $\{A_{n}\}$ is weakly clustered at $S$ in the sense of eigenvalues if, for
   every $\epsilon>0,$ the number of the eigenvalues of $A_{n}$ outside the disc $D(S,\epsilon)$ is bounded by a constant
   $q_{\epsilon},$ possibly depending of $\epsilon$, but independent of $n$. In order words,
   \begin{equation}\label{10n}
    q_{\epsilon}(n,S):=\#\{j:\lambda_{j}(A_{n})\notin D(S,\epsilon)\}=o(n),\text{\,\,\,as\,\,\,}n\rightarrow\infty.
   \end{equation}
   \end{definition}

   If $\{A_{n}\},$ is weakly clustered as $S,$ and $S$ is not connected then its disjoint parts are called to be
   sub-clustered. Finally, if we replace eigenvalues with singular values in relations $(\ref{10n}),$ we
   obtain the definitions of a matrix-sequence weakly clustered at a closed subset of $\mathbb{C},$ in the sense of
   singular values.

   \begin{remark}\label{7r}
   Let $\{A_{n}\},$ be a sequence of matrices $\{A_{n}\}$ of order $n.$ If $\{A_{n}\}\sim_{\lambda}(\theta,G^{d}),$ with $\{A_{n}\},$
   $\theta,$ and $G,$ as in Definition $\ref{5d},$ then $\{A_{n}\},$ is weakly clustered at $\mathcal{ER}(\theta),$ in the
   sense of the eigenvalues. Furthermore, it is clear that
    $\{A_{n}\}\sim_{\lambda}(\theta,G^{d}),$ with $\theta=r,$ equal to a constant function, is equivalent to saying that
   $\{A_{n}\},$ is weakly clustered at $r\in\mathbb{C},$ in the sense of the eigenvalues. For more results and relations
   between the notions of equal distribution, equal localization, spectral distribution, spectral clustering etc...,
   (see \cite{25sss}, Section $4$).
   \end{remark}

    Using Theorem $\ref{19t},$ in \cite{en} the author proved the following result.
   \begin{theorem}\label{21t}\cite{en}
   Let $f\in L^{\infty}(\mathbb{T})$ and let $g$ be a non-negative integer. If $f$ is real-valued, then
    $\{T_{n,g}(f)\}\sim_{\lambda}(\theta_{g}f,\mathbb{T})$ where
    \begin{equation}\label{22}
    \theta_{g}=\left\{
                  \begin{array}{ll}
                   1, & \hbox{if $g=1$;} \\
                   0, & \hbox{otherwise.}
                     \end{array}
                  \right.
    \end{equation}
   $\{T_{n,g}(f)\}$ is weakly clustered at $\mathcal{ER}(\theta_{g}f),$ and $\mathcal{ER}(\theta_{g}f),$ strongly attracts the
   spectra of $\{T_{n,g}(f)\}$ with infinite order of attraction for any of its points.
   \end{theorem}

   \begin{remark}\label{8r}
   It comes from Theorem $\ref{21t}$ that every matrix-sequence $\{T_{n(k),g}(f)\}_{k}$ such that, $\min n_{j}(k)\rightarrow\infty$ is
   weakly clustered at $\mathcal{ER}(\theta^{(g)}_{f})$ in the sense of the eigenvalues.
   \end{remark}

    \begin{remark}
   It is not hard to see that $\theta_{f}^{(g)},$ defined in (\cite{en}, relation $(26)$) equals to $\theta_{g}f,$
   given by $(\ref{22}).$
     \end{remark}

   We introduce another interesting notion concerning the eigenvalues of a matrix sequence.

   \begin{definition}\label{9l}
   Let $W$ be a compact subset of $\mathbb{C}.$ The area of $W,$ denoted $\mathcal{A}rea(W),$ is defined by
   \begin{equation*}
    \mathcal{A}rea(W):=\mathbb{C}\setminus U,
   \end{equation*}
   where $U$ is the "unique" unbounded connected component of $\mathbb{C}\setminus W.$
   \end{definition}

    This section ends with the vector space of finite dimension given by relation $(\ref{13})$ which plays a crucial role in
    establishing the proof of Proposition $\ref{28p}$ (a main tool in proving our original result, namely, Theorem $\ref{23t}$).\\

   Let $\mathcal{V}_{n}(z),$ be the subspace of $\mathcal{H}^{2},$ spanned by the set of monomials of degree "less than"
   $z^{n},$ that is,
   \begin{equation}\label{13}
    \mathcal{V}_{n}(z)=span\{z^{j},j=0,1,...,n-1\}.
   \end{equation}

   This is the idea to be used in the proof of Proposition $\ref{28p}.$ In the case $g=1,$ see \cite{35sss,8sss} for a detail to several
   variables (multilevel case) and matrix-valued functions. In addition, we recall that the asymptotic distribution of the
   eigenvalues and singular values of a sequence of Toeplitz matrices has been deeply studied in the last century, and strictly
   depends on the generating functions (see \cite{8sss,35sss,38sss} and references therein).\\

   Armed with the above definitions and notions, we are ready to state the main tools that we shall use for the proof of our
   original contribution (Theorem $\ref{23t}$).\\

   \section{Spectral distribution results for the sequences of products\\
   $\{T_{n,g}(f_{1})T_{n,g}(f_{2})\}_{n\in\mathbb{N}}$}\label{III}

   This section deals with the spectral distribution in the eigenvalues sequence of the products of $g$-Toeplitz structures,
   $\{T_{n,g}(f_{1})T_{n,g}(f_{2})\},$ in the case where the symbols $f_{1},f_{2}\in L^{\infty}(G^{d}).$ For the sake of readability,
   we focus our attention on the case $d=1.$ The case where $d>1$ will be the subject of our future investigations.\\

   Let $f\in L^{1}(G),$ where $G\cong\mathbb{T}=(-\pi,\pi].$ We recall that a matrix $T_{n,g}(f),$ is called $g$-Toeplitz
   if its entries obey the rule
   \begin{equation}\label{10}
    T_{n,g}(f)=\left[\hat{f}_{r-gs}\right]_{r,s=0}^{n-1},
   \end{equation}
   where the indices, $r-gs,$ are not reduced modulus $n$ as in the circulant case. In analogy with the case of
   $g=1,$ $\{a_{k}\}_{k},$ is the sequence of Fourier coefficients of a Lebesgue integrable function $f$ defined over the
   domain $G,$ i.e., $a_{k}=\widehat{f}_{k},$ defined by relation $(\ref{1}),$ with $d=1.$ Denoting by $T_{n}(f),$ the classical
   Toeplitz matrix also generated by the symbol $f,$ that is, $T_{n}(f)=\left[\hat{f}_{r-s}\right]_{r,s=0}^{n-1},$ the
   authors (\cite{nss}, page $12$) proved that for $n$ and $g$ generic, the following equality holds
   \begin{equation}\label{11}
    T_{n,g}(f)=T_{n}(f)[\widehat{Z}_{n,g}|0]+[0|\mathcal{T}_{n,g}],
   \end{equation}
   where $\mathcal{T}_{n,g}\in\mathbb{C}^{n\times(n-\mu_{g})}$ ($\mu_{g}=\lceil\frac{n}{g}\rceil$) is the matrix $T_{n,g}(f),$ defined
   in relation $(\ref{10})$ by considering only the $n-\mu_{g}$ last columns of $T_{n,g}(f)$ and $\widehat{Z}_{n,g}$ is the submatrix of
   $Z_{n,g}$ defined in relation $(\ref{12})$ by considering only the $\mu_{g}$ first columns of $Z_{n,g}$, where
    \begin{equation}\label{12}
    Z_{n,g}=\left[\delta_{r-gs}\right]_{r,s=0}^{n-1}\text{\,\,\,and\,\,\,}\delta_{k}=\left\{
                                                                                            \begin{array}{ll}
                                                                                             1,&\hbox{if $k\equiv0\mod n;$}\\
                                                                                               0,&\hbox{otherwise.}
                                                                                               \end{array}
                                                                                             \right.
   \end{equation}

    The following results play a crucial role in our analysis.

   \begin{lemma}\label{15l}\cite{en,nss}
    Let $\widehat{Z}_{n,g},$ be the submatrix defined in relation $(\ref{12})$ by considering only the
    $\mu_{g}=\lceil\frac{n}{g}\rceil,$ first columns, then

    \begin{equation}\label{15}
    \|[\widehat{Z}_{n,g}|0]\|=1.
    \end{equation}
   \end{lemma}

     \begin{lemma}\label{16l}\cite{en}
    For every $f\in L^{\infty}(G),$ the matrix-sequence $\{T_{n,g}(f)\},$ is uniformly bounded by a positive constant,
    $\widehat{C}=\|f\|_{L^{\infty}(G)},$ independent of $n.$
   \end{lemma}

    \begin{lemma}\label{17l}\cite{en}
    Let $f\in L^{\infty}(G),$ then the following relation holds
    \begin{equation}\label{18}
    \|[0|\widetilde{\mathcal{T}}_{n,g}]\|_{1}=o(n),\text{\,\,\,\,\,\,\,\,\,\,\,\,}n\rightarrow\infty.
    \end{equation}
   \end{lemma}

    Furthermore, Theorem $\ref{19t},$ based on a Mirski theorem (see \cite{2sss}, Proposition III, Section $5.3$), establishes
   a link between distribution of non-Hermitian perturbations of Hermitian matrix-sequences and distribution of the original
   sequence.

   \begin{theorem}\label{19t}(\cite{sss}, Theorem $3.4$)
   Let $\{Y_{n}\}$ and $\{Z_{n}\}$ be two matrix-sequences, where $Y_{n}$ is Hermitian and $X_{n}=Y_{n}+Z_{n}.$ Assume
   further that $\{Y_{n}\},$ is distributed as $(f,G),$ in the sense of the eigenvalues, where $G$ is of finite and
   positive Lebesgue measure, both $\{Y_{n}\}$ and $\{Z_{n}\},$ are uniformly bounded by a positive constant $\widehat{C},$
   independent of $n,$ and $\|Z_{n}\|_{1}=o(n)$, $n\rightarrow\infty.$ Then $f$ is real-valued and $\{X_{n}\}$ is
   distributed as $(f,G),$ in the sense of the eigenvalues. In particular, $\{X_{n}\}$ is weakly clustered at $\mathcal{ER}(f),$
   and $\mathcal{ER}(f),$ strongly attracts the spectra of $\{X_{n}\},$ with an infinite order of attraction for any of its
   points.\\
    For definitions and more details concerning weak attraction/strong attraction, with a finite (or infinite) order
    of attraction we refer to \cite{sss}.
   \end{theorem}

   The following theorem is an important tool that we shall use to prove Lemma $\ref{20l}$ which helps in the proof of
   Theorem $\ref{23t}$.

   \begin{theorem}\label{20t} (\cite{sss}, Theorem $4.1$)
   Let $\{X_{n}\}$ be a matrix sequence, with $X_{n}$ of size $d_{n},$ tending to infinity and $S$ a subset of $\mathbb{C}$.
   If
   \begin{description}
     \item[(i1)] $S$ is a compact subset and $\mathbb{C}\setminus S$ is connected;
     \item[(i2)] the matrix sequence $\{X_{n}\}$ is weakly clustered at $S;$
     \item[(i3)] the spectrum $\Lambda_{n}$ of $X_{n},$ is uniformly bounded, i.e., $\lambda\leq C$, $\lambda\in\Lambda_{n},$
     for all n, and for some positive constant $C,$ independent of $n;$
     \item[(i4)] there exists a measurable function $h\in L^{\infty}(G),$ having positive and finite Lebesgue measure, such that,
    for every positive integer $l,$ we have: $\underset{n\rightarrow\infty}{\lim}\frac{tr(X_{n}^{l})}{n}=\frac{1}{mes(G)}
    \int_{G}h^{l}(t)dt,$ that is, relation $(\ref{5})$ holds with F being any polynomial of an arbitrary fixed degree; If, further
     \item[(i5)] $\mathcal{ER}(h)$ is contained in $S;$\\
     then relation $(\ref{5})$ is true for every continuous function F with bounded support, which is holomorphic in the interior of
    $S.$\\
     If, in addition, it is also true that the interior of $S$ is empty then
    \begin{equation*}
    \{X_{n}\}\sim_{\lambda}(h,G).
    \end{equation*}
   \end{description}
   \end{theorem}
   Using Theorem $\ref{20t}$ we prove the following lemma, which is a variation of \cite{29sss}, Theorem 4.4.
   \begin{lemma}\label{20l}
   Let $\{X_{n}\}$ be a matrix sequence, with $X_{n}$ of size $d_{n},$ tending to infinity. If
   \begin{description}
     \item[(i1)] the spectrum $\Lambda_{n}$ of $X_{n},$ is uniformly bounded, i.e., $\lambda\leq C$, $\lambda\in\Lambda_{n},$
     for all n, and for some positive constant $C,$ independent of $n;$
     \item[(i2)] there exists a measurable function $h\in L^{\infty}(G),$ having positive and finite Lebesgue measure, such that,
    for every positive integer $l,$ we have: $\underset{n\rightarrow\infty}{\lim}\frac{tr(X_{n}^{l})}{n}=\frac{1}{mes(G)}
    \int_{G}h^{l}(t)dt;$
     \item[(i3)] there exist constants $\widehat{C}_{1}>0,$ $\widehat{C}_{2}>0,$ independent of $n,$ and a real number,
     $q\in[1,\infty),$ independent of n, such that, $\|p(X_{n})\|_{q}^{q}\leq n\left\{\frac{\widehat{C}_{1}}{mes(G)}\int_{G}|p(h(t))|^{q}dt
     +\widehat{C}_{2}\right\}$ for every fixed polynomial $p,$ independent of $n,$ and for every $n$ large enough;\\

    then the matrix sequence $\{X_{n}\},$ is weakly clustered at $\mathcal{A}rea(\mathcal{ER}(h))$
    (see Definition $\ref{9l}$) and relation $(\ref{5})$ is true for every continuous function F with bounded support which is
    holomorphic in the interior of $\mathcal{A}rea(\mathcal{ER}(h)).$ \\
    If, in addition,
     \item[(i4)] $\mathcal{ER}(h)$ does not disconnect the complex field and the interior of $\mathcal{ER}(h)$ is empty, then
    \begin{equation*}
    \{X_{n}\}\sim_{\lambda}(h,G).
    \end{equation*}
   \end{description}
   \end{lemma}

   \begin{proof}
    Since $h\in L^{\infty}(G),$ then $\mathcal{ER}(h)$ is bounded. Using the fact that the essential range is always closed,
    it is obvious that $\mathcal{ER}(h)$ is compact. Let us set $S=\mathcal{A}rea(\mathcal{ER}(h)).$ Our aim is to prove that
    $S$ is a weak cluster for the
    spectra of $\{X_{n}\}.$ Using item (\textbf{i1}) of Lemma $\ref{20l},$ it is easy to see that all the eigenvalues of $X_{n},$ for
    every $n\in\mathbb{N},$ are contained in the compact set $K_{C}=\{x\in\mathbb{C}:|x|\leq C\}.$ This shows that $K_{C}$ is a
    strong cluster for the spectra of $\{X_{n}\}.$ Moreover $C$ can be chosen such that $K_{C}$ contains $S$. Therefore, we will
    have proven that $S$ is a weak cluster for $\{X_{n}\}$ if we prove that, for every $\epsilon>0,$ the compact set
    $K_{C}\setminus D(S,\epsilon)$ contains at most only $o(n)$ eigenvalues, where $D(S,\epsilon)=\underset{x\in S}{\bigcup}D(x,\epsilon)$.
    Using the property of compact sets, for any $\delta>0,$ there exists a finite covering of $K_{C}\setminus D(S,\epsilon)$
    made of balls $D(x,\delta),$ $x\in K_{C}\setminus S$ with $D(x,\delta)\cap S=\emptyset$, and so, it suffices to show that, for
    a particular $\delta,$ at most o(n) eigenvalues lie in $D(x,\delta).$ Let $Q(t)$ be the characteristic function of the compact set
    $\overline{D(x,\delta)}$ and let $\epsilon>0,$ using the Mergelyan's theorem there exists a polynomial $Q_{\epsilon}$ such that $|Q(t)-Q_{\epsilon}(t)|\leq\epsilon$ for every $t\in\overline{D(x,\delta)}\cup S.$ Putting $\gamma_{n}(x,\delta)$ the number of
    eigenvalues of $X_{n}$ belonging to $D(x,\delta),$ let $p\in[1,\infty)$ and $q$ its conjugate, that is, $\frac{1}{p}+\frac{1}{q}=1.$
    A combination of definitions of $Q$ and $\gamma_{n}(x,\delta),$ the approximation property of $Q_{\epsilon}$ and the H\"{o}lder
    inequality gives
    \begin{eqnarray}\label{a11}
    % \nonumber to remove numbering (before each equation)
      \notag(1-\epsilon)\gamma_{n}(x,\delta) &\leq& \underset{k=1}{\overset{n}\sum}Q(\lambda_{k})|Q_{\epsilon}(\lambda_{k})| \\
      \notag &\leq& \left(\underset{k=1}{\overset{n}\sum}Q^{p}(\lambda_{k})\right)^{\frac{1}{p}}
      \left(\underset{k=1}{\overset{n}\sum}|Q_{\epsilon}(\lambda_{k})|^{q}\right)^{\frac{1}{q}} \\
      \notag &=& \left(\underset{k=1}{\overset{n}\sum}Q(\lambda_{k})\right)^{\frac{1}{p}}
      \left(\underset{k=1}{\overset{n}\sum}|Q_{\epsilon}(\lambda_{k})|^{q}\right)^{\frac{1}{q}} \\
      \notag &=& \left(\gamma_{n}(x,\delta)\right)^{\frac{1}{p}}
      \left(\underset{k=1}{\overset{n}\sum}|Q_{\epsilon}(\lambda_{k})|^{q}\right)^{\frac{1}{q}} \\
             &\leq& \left(\gamma_{n}(x,\delta)\right)^{\frac{1}{p}}\|Q_{\epsilon}(X_{n})\|_{q} \\
             &\leq& \left(\gamma_{n}(x,\delta)\right)^{\frac{1}{p}}\left(\frac{\widehat{C}_{1}n}{mes(G)}\int_{G}|Q_{\epsilon}(h(t))|^{q}dt
     +\widehat{C}_{2}n\right)^{\frac{1}{q}} \\
             &\leq& \left(\gamma_{n}(x,\delta)\right)^{\frac{1}{p}}\left(\widehat{C}_{1}n\epsilon^{q}+\widehat{C}_{2}n\right)^{\frac{1}{q}} \\
             &\leq& \left(\gamma_{n}(x,\delta)\right)^{\frac{1}{p}}n^{\frac{1}{q}}\left(\widehat{C}_{1}\epsilon^{q}
             +\widehat{C}_{2}\right)^{\frac{1}{q}}
    \end{eqnarray}
    where $(20)$ comes from the fact that, for any square matrix, the vector with the moduli of the eigenvalues is weakly-majorized
    by the vector of the singular values \cite{2sss}, estimate $(21)$ follows from assumption \textbf{(i3)} of Lemma $\ref{20l}$
    (which holds for any polynomial of fixed degree), and inequality $(22)$ follows from the approximation properties of
    $Q_{\epsilon}$ over the area delimited by the range of $h.$ Now, using the above estimates along with relation
    $\frac{1}{p}+\frac{1}{q}=1,$ simple computations provide
    \begin{equation*}
        \gamma_{n}(x,\delta)\leq n(1-\epsilon)^{-q}(\widehat{C}_{1}\epsilon^{q}+\widehat{C}_{2}),
    \end{equation*}
    and since $\epsilon$ is arbitrary we get $\gamma_{n}(x,\delta)=o(n).$ Hence, assumptions \textbf{(i1)-(i5)} of Theorem $\ref{20t}$
    hold with $S=\mathcal{A}rea(\mathcal{ER}(h)),$ which is necessarily compact and with connected complement, and consequently the first
    conclusion of Theorem $\ref{20t}$ holds. Finally if $\mathbb{C}\setminus \mathcal{ER}(h)$ is connected and the interior of
    $\mathcal{ER}(h)$ is empty then all the hypotheses of Theorem $\ref{20t}$ are satisfied, therefore the matrix sequence $\{X_{n}\}$
    is distributed in the sense of the eigenvalues as $h$ on its domain $G.$
   \end{proof}

   In the following, we combine both Lemma $\ref{20l}$ and Theorem $\ref{21t},$ to extend the problem studied in \cite{sss},
    regarding the eigenvalues distribution of products of Toeplitz matrices (clustering and attraction) to the case of
   products of $g$-Toeplitz structures $\{T_{n,g}(f_{1})T_{n,g}(f_{2})\},$ where $f_{1}$ and $f_{2}$ are (essentially) bounded
   functions defined over the domain $G$ and $g\geq2.$

    Let $f_{1},f_{2}\in L^{\infty}(G),$ let $n$ and $g$ be two positive integers and let $m$ be a non-negative integer. Designating by
    $P_{m,f_{1}}$ and $P_{m,f_{2}}$ the arithmetic averages of Fourier sums of order $r,$ with $r\leq m,$ of $f_{1}$ and $f_{2},$
    respectively. By definition, $P_{m,f_{1}}$ and $P_{m,f_{2}}$ are trigonometric polynomials of degree less than or equal $m.$ Using
    the space $\mathcal{V}_{n}(z),$ given in $(\ref{13})$ together with the orthogonal projection defined in Section $\ref{I},$ both
    polynomials $P_{m,f_{1}}$ and $P_{m,f_{2}},$ are defined by
     \begin{equation*}
      P_{m,f_{1}}(e^{it})=\underset{l=-m}{\overset{m}\sum}c_{l}\exp(ilt)\text{\,\,\,and\,\,\,}
      P_{m,f_{2}}(e^{it})=\underset{l=-m}{\overset{m}\sum}d_{l}\exp(ilt),\text{\,\,where\,\,}i^{2}=-1.
     \end{equation*}
     Let $f\in L^{\infty}(G),$  we recall that $T_{n,g}(p_{m,f}),$ is the matrix of the $g$-Toeplitz operator $T_{n,g}^{p_{m,f}},$
     in the basis $\mathcal{B}=\{\exp(ilt):l=0,1,...,n-1\},$ where $T_{n,g}^{p_{m,f}}(u)=P_{n,g}(p_{m,f}\cdot u)$, with
     $u\in\mathcal{H}^{2},$ $P_{n,g}(p_{m,f}u)=P_{n}^{\perp}(p_{m,f}\cdot u\circ h_{g})$ is defined as in
     relation $(\ref{01}),$ $h_{g}$ is the mapping from $\widehat{\mathbb{T}}=\left(-\pi/g,\pi/g\right]$ onto $\mathbb{T}$ given
     by $h_{g}(t)=gt,$ and $P_{n}^{\perp}$ is the orthogonal projection onto the space $\mathcal{V}_{n}(z),$ of analytic polynomials
     of degree at most $n.$\\

   \begin{theorem}\label{23t}
   Let $f_{1},f_{2}\in L^{\infty}(G),$ $n$ and $g$ be two positive integers,
    and let $m$ be a non-negative integer. Designating by $P_{m,f_{1}}$ and $P_{m,f_{2}},$ the arithmetic averages of Fourier sums of
    order $r,$ with $r\leq m,$ of $f_{1}$ and $f_{2},$ respectively (for example, see \cite{41sss,3sss} for more details), and
    assuming further that for every $l\in\left\{\left[\frac{m}{g}\right],\left[\frac{m+1}{g}\right],..., \left[\frac{n-1-m}{g}
    \right]\right\},$ it holds: $T_{n,g}^{P_{m,f_{1}}P_{m,f_{2}}}\left(\exp(iglt)-\exp(ilt)\right)=0.$
    Setting $h=f_{1}f_{2},$ and $A_{n,g}(f_{1},f_{2})=T_{n,g}(f_{1})T_{n,g}(f_{2})$, then the matrix-sequence
   $\{A_{n,g}(f_{1},f_{2})\},$ is a weak cluster for $\mathcal{A}rea(\mathcal{ER}(\theta_{g}h)).$ Furthermore,
   $\{A_{n,g}(f_{1},f_{2})\},$ distributes in the sense of the eigenvalues as the function $\theta_{g}h,$ over the domain $G,$
    where $\theta_{g}$ is given by relation $(\ref{22}).$
   \end{theorem}

   First, the proof of Theorem $\ref{23t}$ requires some well known intermediate results. Given two square matrices
   $X$ and $Y,$ of the same size, and two numbers $p,q\in[1,\infty),$ satisfying the relation $\frac{1}{p}+\frac{1}{q}=1.$ The
   H\"{o}lder inequality is given by $\|XY\|_{1}\leq\|X\|_{p}\|Y\|_{q}$ (for more details, see for example \cite{2sss},
   Problem $III.6.2$, and Corollary $IV.2.6$). Specifically, we use in this proof the H\"{o}lder inequalities with $p=1$ and
   $q=\infty,$ which involve the operator norm alone with the trace-norm. That is,
   \begin{equation}\label{24}
    \|XY\|_{1}\leq\|X\|_{1}\|Y\|.
   \end{equation}

   If $p,q\in[1,\infty),$ are conjugate exponents (i.e., $\frac{1}{p}+\frac{1}{q}=1$), and $f\in L^{p}(G),$ $h\in L^{q}(G),$
   a straightforward computation involving the H\"{o}lder inequalities for both Schatten $p$-norms and $L^{p}(G)$-norms shows that,
   $fh\in L^{1}(G),$ and in fact, $\|fh\|_{L^{1}(G)},$ $\|hf\|_{L^{1}(G)}\leq\|f\|_{L^{p}(G)}\|h\|_{L^{q}(G)}.$ This work
   uses the case where $p=1$ and $q=\infty,$ that is,
   \begin{equation}\label{25}
    \|fh\|_{L^{1}(G)}\leq\|f\|_{L^{1}(G)}\|h\|_{L^{\infty}(G)}.
   \end{equation}

   Furthermore, the proof of Theorem $\ref{23t}$ also needs two important results. The first (Proposition $\ref{28p}$) provides
   an estimate of the rank of the matrix $A_{n,g}(f_{1},f_{2})-T_{n,g}(f_{1}f_{2}),$ in the case where $f_{1},f_{2}\in L^{\infty}(G).$
   The second one (Proposition $\ref{29p}$) deals with the trace-norm of $A_{n,g}(f_{1},f_{2})-T_{n,g}(f_{1}f_{2}),$ for
   $f_{1},f_{2}\in L^{\infty}(G),$ which seems to be the crucial point for the proof of Theorem $\ref{23t}.$\\

   Lastly, we recall the following Theorem and Definition which are key for establishing (i3) of Theorem $\ref{20t}.$

   \begin{theorem} (Schur) \label{26t}\cite{2sss}
    For every $n\times n$ matrix $X,$ there exists a unitary $n\times n$ matrix $U,$ such that,
    \begin{equation}\label{27}
    X=UT_{X}U^{'},
    \end{equation}
    where ' denotes the conjugate transpose of a matrix and $T_{X}$ is an $n\times n$ upper triangular matrix whose
    the diagonal elements represent the eigenvalues of $X.$
    \end{theorem}

   \begin{definition}\label{27d}\cite{13sss}
    Let $\{A_{n}\}$ and $\{B_{n}\},$ be two matrix-sequences of size $n.$ The sequences $\{A_{n}\}$ and $\{B_{n}\},$ are
    said to be asymptotically equivalent if,

    \begin{description}
      \item[(c1)] $A_{n}$ and $B_{n},$ are uniformly bounded by a positive constant, $C,$ independent of $n,$ that is,
       \begin{equation*}
        \underset{n\in\mathbb{N}}{\sup}\|A_{n}\|,\text{\,\,}\underset{n\in\mathbb{N}}{\sup}\|B_{n}\|\leq C;
       \end{equation*}
      \item[(c2)] $\frac{1}{n}(A_{n}-B_{n})$ goes to zero in the trace-norm, as $n\rightarrow\infty:$
       \begin{equation*}
        \|A_{n}-B_{n}\|_{1}=o(n),\text{\,\,\,}n\rightarrow\infty.
       \end{equation*}
      \end{description}
      \end{definition}

       Lemma $\ref{31l}$ stated below plays a crucial role in the proof of Proposition $\ref{30p}$ which is the key of
       item \textbf{(i3)} of Lemma $\ref{20l}.$

     \begin{lemma}\label{31l}
      Suppose that $\{X_{n}\}$ and $\{Y_{n}\}$ are two asymptotically equivalent sequences of matrices. Let $P$ be
      any fixed polynomial of degree independent on $n.$ Then the sequences of matrices $\{P(X_{n})\}$ and $P(\{Y_{n})\}$
      are asymptotically equivalent (in the sense of Definition $\ref{27d}$).
     \end{lemma}

     \begin{proof}
      Let $P(z)=\underset{r=0}{\overset{m}\sum}a_{r}z^{r},$ be a fixed polynomial of degree $m,$ independent of $n.$ Setting
      $X_{n}=Y_{n}+D_{n},$ and using the assumption $\{X_{n}\}$ and $\{Y_{n}\},$ are asymptotically equivalent sequences of
      matrices, we see that $D_{n}$ is uniformly bounded by a positive constant independent on $n,$ and $\|D_{n}\|_{1}=o(n),$
      $n\rightarrow\infty.$ Applying the binomial theorem, a straightforward computation shows that $X_{n}^{r}=Y_{n}^{r}+S_{n,r},$
       which implies
      \begin{equation}\label{31ee}
       X_{n}^{r}-Y_{n}^{r}=S_{n,r},  \text{\,\,\,for\,\,every\,\,\,}r=0,1,...,m,
      \end{equation}
       where the matrix $S_{n,r}$ is a sum of several terms each being a product of $Y_{n}$ and $D_{n,}$ but
      containing at least one $D_{n}.$ Multiplying relation $(\ref{31ee})$ side by side by $a_{r},$ summing the resulting
      equation from $r=0,1,...,m,$ and applying the trace norm on the final relation yield
      \begin{equation}\label{32ee}
       \|P(X_{n})-P(Y_{n})\|_{1}=\left\|\underset{r=0}{\overset{m}\sum}a_{r}S_{n,r}\right\|_{1}.
      \end{equation}
      Now, a combination of relation $(\ref{32ee}),$ triangular inequality and H\"{o}lder inequality several times,
      and after rearranging terms provide
       \begin{equation}\label{33ee}
       \|P(X_{n})-P(Y_{n})\|_{1}\leq\|D_{n}\|_{1}\underset{r=0}{\overset{m}\sum}|a_{r}|\|\widetilde{S}_{n,r}\|,
      \end{equation}
       where the matrix $\widetilde{S}_{n,r}$ is a finite sum of several terms whose each is a product of $Y_{n}$ and $D_{n,}$ but
      not necessarily containing the matrix $D_{n}.$ Since $Y_{n}$ and $D_{n}$ are uniformly bounded by a positive constant
      independent on $n,$ using the triangular inequality along with the property of the spectral norm, it is easy to see that
      $\|\widetilde{S}_{n,r}\|\leq\widehat{C}_{r},$ where $\widehat{C}_{r}$ is a positive constant dependent on $r,$ but
      independent on $n.$ Using this fact, equality $\|D_{n}\|_{1}=o(n)$ and estimate $(\ref{33ee}),$ we obtain
      \begin{equation}\label{34ee}
       \|P(X_{n})-P(Y_{n})\|_{1}\leq\widehat{C}_{m}\cdot o(n),\text{\,\,\,\,}n\rightarrow\infty,
      \end{equation}
      where $\widehat{C}_{m}=\widehat{C}_{m}(P)$ is some positive constant independent on $n.$\\

      On the other hand, the matrices $X_{n}$ and $Y_{n}$ are uniformly bounded by a positive constant independent on $n,$
      using the property of the operator norm, it is not hard to see that for every $r=0,1,...,m,$ $X_{n}^{r}$ and $Y_{n}^{r}$ are
      uniformly bounded by a positive constant independent on $n.$ This fact together with the triangular inequality
      show that both $P(X_{n})$ and $P(Y_{n})$  are uniformly bounded by a positive constant independent on $n,$ and the proof
      of Lemma $\ref{31l}$ is established.
      \end{proof}

      The following results, namely Lemmas $\ref{32l}$-$\ref{33l},$ help to prove both Propositions $\ref{28p}$ and $\ref{29p}.$

    \begin{lemma}\label{32l}
    Let $f\in L^{\infty}(G),$ then it holds
     \begin{equation*}
     \|T_{n,g}(f)\|_{1}\leq\frac{n}{2\pi}\|f\|_{L^{\infty}(G)}+o(n),\text{\,\,\,}n\rightarrow\infty.
     \end{equation*}
    \end{lemma}

    \begin{proof}
     First, using relation $(\ref{11}),$ we have $T_{n,g}(f)=T_{n}(f)[\widehat{Z}_{n,g}|0]+[0|\mathcal{T}_{n,g}|].$
     A combination of triangular inequality, H\"{o}lder inequalities $(\ref{24})$-$(\ref{25}),$ Corollary $4.2$ in
     \cite{1sss,31sss}, Lemma $\ref{15l},$ and relation $(\ref{18})$ gives
     \begin{equation*}
    \|T_{n,g}(f)\|_{1}\leq\|T_{n}(f)[\widehat{Z}_{n,g}|0]\|_{1}+\|[0|\mathcal{T}_{n,g}|]\|_{1}\leq
     \|T_{n}(f)\|_{1}\|[\widehat{Z}_{n,g}|0]\|+\|[0|\mathcal{T}_{n,g}|]\|_{1}\leq\frac{n}{2\pi}\|f\|_{L^{\infty}(G)}
     +o(n),\text{\,\,}n\rightarrow\infty.
     \end{equation*}
    \end{proof}

    \begin{lemma}\label{33l}
    If $f_{1},f_{2}\in L^{\infty}(G)$ and $A_{n,g}(f_{1},f_{2})=T_{n,g}(f_{1})T_{n,g}(f_{2}),$ then the matrix
     sequence $\{A_{n,g}(f_{1},f_{2})\},$ is uniformly bounded by a positive constant, $\widehat{C},$ independent on $n.$
    \end{lemma}

       \begin{proof}
     Since $f_{1},f_{2}\in L^{\infty}(G),$ using Lemma $\ref{16l},$ the matrix sequences $\{T_{n,g}(f_{1})\}$ and
     $\{T_{n,g}(f_{2})\}$ are uniformly bounded by $\|f_{1}\|_{L^{\infty}(G)}$ and $\|f_{2}\|_{L^{\infty}(G)},$ respectively.
     So, the sequence of matrices $\{A_{n,g}(f_{1},f_{2})\}$ is uniformly bounded by a positive constant $\widehat{C}=
     \|f_{1}\|_{L^{\infty}(G)}\|f_{2}\|_{L^{\infty}(G)}$ which is independent on $n.$
    \end{proof}

    \begin{proposition}\label{28p}
    Let $f_{1},f_{2}\in L^{\infty}(G)$ and let $m$ be a positive integer. Denoting by $P_{m,f_{1}}$ and $P_{m,f_{2}},$
    the arithmetic averages of Fourier sums of order $r,$ with $r\leq m,$ of $f_{1}$ and $f_{2},$ respectively, and assuming
    further that for every $l\in\left\{\left[\frac{m}{g}\right],\left[\frac{m+1}{g}\right],..., \left[\frac{n-1-m}{g}\right]
    \right\},$ $T_{n,g}^{P_{m,f_{1}}P_{m,f_{2}}}\left(\exp(iglt)-\exp(ilt)\right)=0.$ Then it holds
    \begin{equation}\label{34}
    rank\left(A_{n,g}(P_{m,f_{1}},P_{m,f_{2}})-T_{n,g}(P_{m,f_{1}}P_{m,f_{2}})\right)\leq 2\left[\frac{m}{g}\right].
    \end{equation}
    \end{proposition}

     \begin{proof}
      First, the polynomials $P_{m,f_{1}}$ and $P_{m,f_{2}}$ are defined by
     \begin{equation*}
      P_{m,f_{1}}(e^{it})=\underset{l=-m}{\overset{m}\sum}c_{l}\exp(ilt)\text{\,\,\,and\,\,\,}
      P_{m,f_{2}}(e^{it})=\underset{l=-m}{\overset{m}\sum}d_{l}\exp(ilt),\text{\,\,where\,\,}i^{2}=-1.
     \end{equation*}
     Let $f\in L^{\infty}(G),$  we recall that $T_{n,g}(P_{m,f}),$ is the matrix of the $g$-Toeplitz operator $T_{n,g}^{P_{m,f}},$
     in the basis $\mathcal{B}=\{\exp(ilt):l=0,1,...,n-1\},$ where $T_{n,g}^{P_{m,f}}(u)=P_{n,g}(P_{m,f}\cdot u)$, with
     $u\in\mathcal{H}^{2},$ $P_{n,g}(P_{m,f}\cdot u)=P_{n}^{\perp}(P_{m,f}\cdot u\circ h_{g}),$ is defined as in
     relation $(\ref{01}),$ $h_{g}$ is the mapping from $\widehat{\mathbb{T}}=\left(-\pi/g,\pi/g\right]$ onto $\mathbb{T}$ given
     by $h_{g}(t)=gt,$ and $P_{n}^{\perp}$ is the orthogonal projection onto the space $\mathcal{V}_{n}(z),$ of analytic polynomials
     of degree at most $n.$\\

     Now, for every $l\in\left\{\left[\frac{m}{g}\right],\left[\frac{m+1}{g}\right],...,\left[\frac{n-1-m}{g}\right]\right\}$
     (where $[x]$ designates the greatest integer less than $x$), it is not hard to prove that the function $\mathcal{B}_{l}$
     defined by $\mathcal{B}_{l}(e^{it})=(P_{m,u})(e^{it})\exp(iglt),$ belongs to $\mathcal{V}_{n}(z),$ so using the
     definition of $T_{n,g}^{P_{m,u}},$ a straightforward calculation provides
     \begin{equation*}
      T_{n,g}^{P_{m,f_{2}}}(\exp(ilt))=P_{n,g}\left(P_{m,f_{2}}(e^{it})\cdot\exp(ilt)\right)=
     P_{n}^{\perp}\left(P_{m,f_{2}}(e^{it})\cdot\exp(iglt)\right)=P_{m,f_{2}}(e^{it})\cdot\exp(iglt),
     \end{equation*}
     using this we have
     \begin{equation*}
      T_{n,g}^{P_{m,f_{1}}}T_{n,g}^{P_{m,f_{2}}}(\exp(ilt))=T_{n,g}^{P_{m,f_{1}}}\left(P_{m,f_{2}}(e^{it})\cdot\exp(iglt)\right)=
     P_{n,g}\left(P_{m,f_{1}}(e^{it})P_{m,f_{2}}(e^{it})\cdot\exp(iglt)\right)=
     \end{equation*}
     \begin{equation*}
      T_{n,g}^{P_{m,f_{1}}P_{m,f_{2}}}(\exp(iglt)).
     \end{equation*}
     Combining this together with the linearity of the operator $T_{n,g}^{P_{m,f_{1}}P_{m,f_{2}}},$ and the hypothesis given
     in Proposition $\ref{28p}$ results in
     \begin{equation}\label{35}
      \left(T_{n,g}^{P_{m,f_{1}}}T_{n,g}^{P_{m,f_{2}}}-T_{n,g}^{P_{m,f_{1}}P_{m,f_{2}}}\right)(\exp(ilt))=
      T_{n,g}^{P_{m,f_{1}}P_{m,f_{2}}}\left(\exp(iglt)-\exp(ilt)\right)=0,
     \end{equation}
     for every $l\in\left\{\left[\frac{m}{g}\right],\left[\frac{m+1}{g}\right],...,\left[\frac{n-1-m}{g}\right]
     \right\}.$ Relation $(\ref{35})$ means that the image of the operator $T_{n,g}^{P_{m,f_{1}}}T_{n,g}^{P_{m,f_{2}}}-
     T_{n,g}^{P_{m,f_{1}}P_{m,f_{2}}},$ is generated by the image of the set $\left\{\exp(ilt):l=\left[\frac{n-m}{g}\right],
     ,...,\left[\frac{n-1}{g}\right],\text{\,or,\,\,}l=0,...,\left[\frac{m-1}{g}\right]\right\},$
     which is of cardinality less than or equal to $2\left[\frac{m}{g}\right].$ Furthermore, $A_{n,g}(P_{m,f_{1}},P_{m,f_{2}})-
     T_{n,g}(P_{m,f_{1}}P_{m,f_{2}}),$ is the matrix related to the operator $T_{n,g}^{P_{m,f_{1}}}T_{n,g}^{P_{m,f_{2}}}-
     T_{n,g}^{P_{m,f_{1}}P_{m,f_{2}}},$ in the basis $\{\exp(ilt):l=0,1,...,n-1\}.$ Hence, we find that the rank of $A_{n,g}(P_{m,f_{1}},P_{m,f_{2}})-T_{n,g}(P_{m,f_{1}}P_{m,f_{2}})$ is smaller than $2\left[\frac{m}{g}\right],$ and
     Proposition $\ref{28p}$ is proved.
    \end{proof}

    \begin{proposition}\label{29p}
    Let $f_{1},f_{2}\in L^{\infty}(G)$ and let $A_{n,g}(f_{1},f_{2})=T_{n,g}(f_{1})T_{n,g}(f_{2}),$ $h=f_{1}f_{2}.$ Assume that
    $m$ be a non negative integer. Denoting by $P_{m,f_{1}}$ and $P_{m,f_{2}},$ the arithmetic averages of Fourier sums of order $r,$
    with $r\leq m,$ of $f_{1}$ and $f_{2},$ respectively, and assuming further that $T_{n,g}^{P_{m,f_{1}}P_{m,f_{2}}}
    \left(\exp(iglt)-\exp(ilt)\right)=0,$ for every $l\in\left\{\left[\frac{m}{g}\right],\left[\frac{m+1}{g}\right],...,
    \left[\frac{n-1-m}{g}\right]\right\}.$ Then
    \begin{equation}\label{34}
    \|A_{n,g}(f_{1},f_{2})-T_{n,g}(h)\|_{1}=o(n),\text{\,\,\,}n\rightarrow\infty.
    \end{equation}
    \end{proposition}

    \begin{proof}
     Since $f_{1},f_{2}\in L^{\infty}(G),$ with $A_{n,g}(f_{1},f_{2})=T_{n,g}(f_{1})T_{n,g}(f_{2})$ and $h=f_{1}f_{2},$ to
     estimate the Schatten $1$-norm of the matrix $A_{n,g}(P_{m,f_{1}},P_{m,f_{2}})-T_{n,g}(P_{m,f_{1}}P_{m,f_{2}}),$ we use some
     well known results from the approximation theory alone with Proposition $\ref{28p}.$ Now, $P_{m,f_{1}}$ and $P_{m,f_{2}},$
     being the arithmetic averages of Fourier sums of order $r,$ with $r\leq m,$ of $f_{1}$ and $f_{2},$ respectively, it is obvious
     that the sequences of polynomials $\{P_{m,f_{1}}\}_{m}$ and $\{P_{m,f_{2}}\}_{m},$ converge in $L^{1}$-norm to $f_{1}$ and $f_{2},$
     respectively, as $m$ goes to infinity and that
     \begin{equation}\label{37}
      \|P_{m,f_{1}}\|_{L^{\infty}}\leq\|f_{1}\|_{L^{\infty}}\text{\,\,\,and\,\,\,}\|P_{m,f_{2}}\|_{L^{\infty}}\leq\|f_{2}\|_{L^{\infty}}.
     \end{equation}
     Using the triangular inequality several times, simple computations yield
     \begin{equation*}
       \|\left(A_{n,g}(f_{1},f_{2})-T_{n,g}(h)\right)\|_{1}=\|(A_{n,g}(f_{1},f_{2})-T_{n,g}(P_{m,f_{1}})T_{n,g}(f_{2}))+
       (T_{n,g}(P_{m,f_{1}})T_{n,g}(f_{2})- T_{n,g}(P_{m,f_{1}})\times
     \end{equation*}
     \begin{equation*}
      T_{n,g}(P_{m,f_{2}}))+(T_{n,g}(P_{m,f_{1}})T_{n,g}(P_{m,f_{2}})-T_{n,g}(P_{m,f_{1}}P_{m,f_{2}}))+
      (T_{n,g}(P_{m,f_{1}}P_{m,f_{2}}))-T_{n,g}(h))\|_{1}
     \end{equation*}
     \begin{equation*}
       \leq\|A_{n,g}(f_{1},f_{2})-T_{n,g}(P_{m,f_{1}})T_{n,g}(f_{2})\|_{1}+\|T_{n,g}(P_{m,f_{1}})T_{n,g}(f_{2})-
       T_{n,g}(P_{m,f_{1}})T_{n,g}(P_{m,f_{2}})\|_{1}
     \end{equation*}
     \begin{equation}\label{38}
      +\|T_{n,g}(P_{m,f_{1}})T_{n,g}(P_{m,f_{2}})-T_{n,g}(P_{m,f_{1}}P_{m,f_{2}})\|_{1}+
      \|T_{n,g}(P_{m,f_{1}}P_{m,f_{2}}))-T_{n,g}(h)\|_{1}.
     \end{equation}
     A combination of Lemmas $\ref{16l}$-$\ref{32l},$ H\"{o}lder inequalities $(\ref{24})$-$(\ref{25}),$ and the linearity
     of the $g$-Toeplitz operator related to $T_{n,g}(\cdot)$ gives
     \begin{eqnarray}\label{39}
     % \nonumber to remove numbering (before each equation)
       \notag \|A_{n,g}(f_{1},f_{2})-T_{n,g}(P_{m,f_{1}})T_{n,g}(f_{2})\|_{1} &\leq&\|T_{n,g}(f_{1})-
     T_{n,g}(P_{m,f_{1}})\|_{1}\|T_{n,g}(f_{2})\|  \\
        &\leq& \left(\frac{n}{2\pi}\|f_{1}-P_{m,f_{1}}\|_{L^{1}}+o(n)\right)\|f_{2}\|_{L^{\infty}},\text{\,\,\,}
      n\rightarrow\infty,
     \end{eqnarray}
     \begin{eqnarray}\label{40}
     % \nonumber to remove numbering (before each equation)
       \notag \|T_{n,g}(P_{m,f_{1}})T_{n,g}(f_{2})-T_{n,g}(P_{m,f_{1}})T_{n,g}(P_{m,f_{2}})\|_{1} &\leq&\|T_{n,g}(f_{2})-
     T_{n,g}(P_{m,f_{2}})\|_{1}\|T_{n,g}(P_{m,f_{1}})\|  \\
       \notag &\leq& \left(\frac{n}{2\pi}\|f_{2}-P_{m,f_{2}}\|_{L^{1}}+o(n)\right)\|P_{m,f_{1}}\|_{L^{\infty}},\text{\,\,}
      n\rightarrow\infty,\\
        &&
     \end{eqnarray}
     \begin{equation}\label{41}
      \|T_{n,g}(P_{m,f_{1}}P_{m,f_{2}})-T_{n,g}(h)\|_{1} \leq\frac{n}{2\pi}\|P_{m,f_{1}}P_{m,f_{2}}-h\|_{L^{1}}
     ,\text{\,\,\,}n\rightarrow\infty.
     \end{equation}
     Combining estimates $(\ref{39})$-$(\ref{41}),$ we see that the sum of the first, second and fourth terms in relation
     $(\ref{38})$ equals to $n\left(\alpha(m)+\frac{o(n)}{n}\right),$ as $n\rightarrow\infty.$ Since the Cesaro operator
     converges to the identity operator in $L^{1}(G)$-topology, we have
     \begin{equation}\label{41a}
      \underset{m\rightarrow\infty}{\lim}\alpha(m)=0.
     \end{equation}
      Now, let us analyze the term $\|T_{n,g}(P_{m,f_{1}})T_{n,g}(P_{m,f_{2}})-T_{n,g}(P_{m,f_{1}}P_{m,f_{2}})\|_{1},$ of
     estimate $(\ref{38}).$ In fact, since the trace-norm is bounded by the rank times the spectral norm, a combination of
     triangular inequality together with H\"{o}lder inequalities $(\ref{24})$-$(\ref{25}),$ Lemmas $\ref{16l}$ and $\ref{33l},$
     Proposition $\ref{28p},$ and inequality $(\ref{37})$ gives
    \begin{eqnarray}\label{42}
    % \nonumber to remove numbering (before each equation)
      \notag \|T_{n,g}(P_{m,f_{1}})T_{n,g}(P_{m,f_{2}})-T_{n,g}(P_{m,f_{1}}P_{m,f_{2}})\|_{1}&\leq& 2\left[\frac{m}{g}\right]
      \|T_{n,g}(P_{m,f_{1}})T_{n,g}(P_{m,f_{2}})-T_{n,g}(P_{m,f_{1}}P_{m,f_{2}})\| \\
      \notag &\leq& 2\left[\frac{m}{g}\right]\left(\|T_{n,g}(P_{m,f_{1}})T_{n,g}(P_{m,f_{2}})\|+\|T_{n,g}(P_{m,f_{1}}P_{m,f_{2}})
      \|\right) \\
      \notag &\leq& 2\left[\frac{m}{g}\right]\left(\|P_{m,f_{1}}\|_{L^{\infty}}\|P_{m,f_{2}}\|_{L^{\infty}}
      +\|P_{m,f_{1}}P_{m,f_{2}}\|_{L^{\infty}}\right) \\
       &\leq& 4\left[\frac{m}{g}\right]\|P_{m,f_{1}}\|_{L^{\infty}}\|P_{m,f_{2}}\|_{L^{\infty}}
       \notag \leq 4\left[\frac{m}{g}\right]\|f_{1}\|_{L^{\infty}}\|f_{2}\|_{L^{\infty}},\\
       & &
    \end{eqnarray}
     for each $m\in\mathbb{N}.$ Setting $\gamma=4\|f_{1}\|_{L^{\infty}}\|f_{2}\|_{L^{\infty}},$ and combining $(\ref{42})$ with
     $n\cdot\alpha(m)+o(n)$ (given above), estimate $(\ref{38})$ becomes
     \begin{equation}\label{43}
     \|\left(A_{n,g}(f_{1},f_{2})-T_{n,g}(h)\right)\|_{1}\leq n\cdot\alpha(m)+\gamma\cdot\left[\frac{m}{g}\right]+
      o(n),\text{\,\,\,}n\rightarrow\infty,
     \end{equation}
     for all $m\in\mathbb{N}.$ Let $\epsilon>0,$ according to relation $(\ref{41a}),$ there exists $m_{0}\in\mathbb{N},$
      $m_{0}\neq0,$ such that,
      \begin{equation}\label{44}
      m\geq m_{0}\Rightarrow \alpha(m)<\frac{\epsilon}{3}.
      \end{equation}
      In way similar, $\underset{n\rightarrow\infty}{\lim}\frac{o(n)}{n}=0,$ implies there exists $n_{0}\in\mathbb{N},$
      $n_{0}\neq0,$ such that,
      \begin{equation}\label{45}
      n\geq n_{0}\Rightarrow \frac{o(n)}{n}<\frac{\epsilon}{3}.
      \end{equation}
       Finally, $\underset{n\rightarrow\infty}{\lim}\frac{\gamma\left[\frac{m}{g}\right]}{n}=0,$ implies there exists
      $n_{1}\in\mathbb{N},$ $n_{1}\neq0,$ such that,
      \begin{equation}\label{46}
      n\geq n_{1}\Rightarrow \frac{\gamma\left[\frac{m}{g}\right]}{n}<\frac{\epsilon}{3}.
      \end{equation}
      Putting $N=\max\{m_{0},n_{0},n_{1}\}+1$ and combining estimates $(\ref{44})$-$(\ref{46})$ along with relation $(\ref{43}),$
      we obtain
      \begin{equation}\label{47}
      n\geq N\Rightarrow \frac{1}{n}\|\left(A_{n,g}(f_{1},f_{2})-T_{n,g}(h)\right)\|_{1}<\epsilon,
      \end{equation}
      which ends the proof of Proposition $\ref{29p}.$
      \end{proof}

      The proof of the following Proposition establishes item $(i3)$ of Lemma $\ref{20l}.$

    \begin{proposition}\label{30p}
    Let $f_{1},f_{2}\in L^{\infty}(G)$ and let $P$ be a fixed polynomial independent of $n.$ Letting $A_{n,g}(f_{1},f_{2})=T_{n,g}(f_{1})T_{n,g}(f_{2}),$ and $h=f_{1}f_{2},$ then
    there exist non negative constants $\widehat{C}_{1}$ and $\widehat{C}_{2}$ independent of $n$ such that,
    \begin{equation}\label{47}
    \|P\left(A_{n,g}(f_{1},f_{2})\right)\|_{1}\leq n\left\{\frac{\widehat{C}_{1}}{mes(G)}   \int_{G}\left|P((\theta_{g}h)(t))\right|dt+\widehat{C}_{2}\right\}+o(n),
    \end{equation}
      for every $n$ large enough.
    \end{proposition}

     Before the proof let us establish the following result.

     \begin{lemma}\label{48l}
     Let $\{X_{n}\}$ and $\{Y_{n}\}$ be two sequences of matrices of order $n,$ and let
     $P(z)=\underset{r=0}{\overset{m}\sum}a_{r}z^{r},$ be any fixed polynomial of degree $m,$ independent of $n.$ Suppose
     that both $\{X_{n}\}$ and $\{Y_{n}\}$ are uniformly bounded by a positive constant $\widehat{C}_{0},$ independent of
     $n,$ and $\|Y_{n}\|_{1}=\alpha\cdot\left\lceil\frac{n}{g}\right\rceil,$ where $\alpha$ is a positive constant independent
     on $n.$ Then $\|P(X_{n}+Y_{n})\|_{1}\leq\|P(X_{n})\|_{1}+\widehat{C}_{m}\cdot \left\lceil\frac{n}{g}\right\rceil,$
     where $\widehat{C}_{m}$ is a positive constant which depends on $\alpha,$ $m$ and $\widehat{C}_{0},$ but independent on $n.$
      \end{lemma}

     \begin{proof}
      Straightforward computations after rearranging terms yield
     \begin{equation}\label{ll}
      P(X_{n}+Y_{n})=\underset{r=0}{\overset{m}\sum}a_{r}(X_{n}+Y_{n})^{r}=a_{0}I_{n}+\underset{r=1}{\overset{m}\sum}a_{r}\{(X_{n})^{r}+W_{n}\},
     \end{equation}
      where the matrix $W_{n}$ is a sum of several terms each being a product of $X_{n}$ and $Y_{n}$, but containing
     at least one $Y_{n}$ (to see this use the binomial theorem applied to matrices to expand $(X+Y)^{r}$), that is,
      a polynomial of fixed degree independent on $n,$ having as variables $X_{n}$ and $Y_{n},$ and whose each term contains
      at least one $Y_{n}.$ In this development we used the convention $X^{0}=I_{n},$ for a square matrix $X$ of
       size $n,$ $I_{n}$ being the identity matrix of order $n.$ Combining the triangular inequality and the H\"{o}lder
       inequality several times and after rearranging terms, equality $(\ref{ll})$ provides
      \begin{equation}\label{ll1}
       \|P(X_{n}+Y_{n})\|_{1}\leq \|P(X_{n})\|_{1}+\|Y_{n}\|_{1}\underset{r=0}{\overset{m}\sum}|a_{r}|\cdot\|\widetilde{W}_{n}\|.
      \end{equation}
      Since $X_{n}$ and $Y_{n}$ are uniformly bounded by a positive constant, $\widehat{C}_{0},$ independent of $n,$ where
      the matrix $\widetilde{W}_{n}$ is also a sum of several terms each being a product of $X_{n}$ and $Y_{n},$ but not necessarily
      containing at least one $Y_{n},$ using again the triangular inequality several times together with the property of operator
      norm, simple calculations show that there exists a positive constant $\widehat{C}_{r},$ that depends $\widehat{C}_{0},$
      but independent on $n$ such that, $\|\widetilde{W}_{n}\|\leq \widehat{C}_{r}.$ Utilizing this, there is a positive constant
      $\widehat{C}_{m}$ which depends on both parameters $m$ and $\widehat{C}_{0},$ but independent of $n,$ that satisfies $\underset{r=0}{\overset{m}\sum}|a_{r}|\|\widetilde{W}_{n}\|\leq \widehat{C}_{m}.$ This fact, alone with the restriction $\|Y_{n}\|_{1}=\alpha\cdot\left\lceil\frac{n}{g}\right\rceil,$ and estimate $(\ref{ll1})$ end the proof of Lemma $\ref{48l}.$
      \end{proof}

     \begin{proof} (of Proposition $\ref{30p}$)
     Using the property that the sequence of Toeplitz matrices with $L^{1}(G)$ symbols belongs to the GLT class together with
     their algebra (see Section $1.2,$ page $8$, in \cite{sss}), it holds
     \begin{equation}\label{48}
     \{P(T_{n}(\theta_{g}h))\}\sim_{\sigma}(P(\theta_{g}h),G).
     \end{equation}
     Since $\theta_{g}h\in L^{\infty}(G)\subset L^{1}(G),$ and $P$ is a fixed polynomial independent of $n,$ it is easy
     to see that the function $P(\theta_{g}h)$ belongs to $L^{1}(G).$ That is, there are non negative constants $\widetilde{m}$
     and $M,$ (with $\widetilde{m}\neq0$ if $P(\theta_{g}h)$ is non null) such that, $\widetilde{m}\leq\frac{1}{mes(G)}\int_{G}
      \left|P((\theta_{g}h)(t))\right|dt\|_{L^{1}(G)}\leq M.$ Taking the function $F$ defined by $F(x)=\left\{
                                                                 \begin{array}{ll}
                                                                   x-\widetilde{m}, & \hbox{if $x\in[0,\widetilde{m}]$,} \\
                                                                   0, & \hbox{for $x\in\mathbb{R}_{0}^{+}\setminus[0,\widetilde{m}],$}
                                                                 \end{array}
                                                               \right.$ it is obvious that $F\in\mathcal{C}_{0}(\mathbb{R}^{+}_{0}).$
    Using this, relation $(\ref{48})$ means that for $n$ large enough, we have
     \begin{equation}\label{49}
      \frac{1}{n}\underset{j=1}{\overset{n}\sum}\sigma_{j}(P(T_{n}(\theta_{g}h)))\leq\underset{\geq0}{\underbrace{\frac{1}{mes(G)}\int_{G}
      \left|P((\theta_{g}h)(t))\right|dt-\widetilde{m}}}\leq\frac{1}{mes(G)}\int_{G}\left|P((\theta_{g}h)(t))\right|dt,
     \end{equation}
     where $\sigma_{j}(P(T_{n}(\theta_{g}h))),$ $j=1,2,...,n,$ are the singular values of $P(T_{n}(\theta_{g}h)).$ Since
     the sum, $\underset{j=1}{\overset{n}\sum}\sigma_{j}(P(T_{n}(\theta_{g}h))),$ is the trace norm of the matrix
      $P(T_{n}(\theta_{g}h)),$ relation $(\ref{49})$ becomes
      \begin{equation}\label{50}
      \|P(T_{n}(\theta_{g}h))\|_{1}\leq\frac{n}{mes(G)}\int_{G}\left|P((\theta_{g}h)(t))\right|dt,
     \end{equation}
      for $n$ sufficiently large.\\

      On the other hand, combining both relations $(\ref{11})$ and $(\ref{18}),$ H\"{o}lder inequalities $(\ref{24})$-$(\ref{25}),$
     triangular inequality, and the definition of the function $\theta_{g}$ (given by relation $(\ref{22})$), we obtain
     \begin{equation*}
      \|T_{n,g}(h))-T_{n}(\theta_{g}h)\|_{1}\leq\|T_{n,g}(h)-T_{n}(h)[\widehat{Z}_{n,g}|0]\|_{1}+\|T_{n}(h)[\widehat{Z}_{n,g}|0]-
      T_{n}(\theta_{g}h)\|_{1}=\|[0|\mathcal{T}_{n,g}]\|_{1}+
     \end{equation*}
      \begin{equation}\label{50e}
      \|T_{n}(h)[\widehat{Z}_{n,g}|0]-T_{n}(\theta_{g}h)\|_{1}=o(n)+\|T_{n}(h)[\widehat{Z}_{n,g}|0]-T_{n}(\theta_{g}h)\|_{1}=
       \left\{
              \begin{array}{ll}
              o(n) & \hbox{if $g=1$,} \\
              o(n)+\|T_{n}(h)[\widehat{Z}_{n,g}|0]\|_{1}& \hbox{for $g>1$,}
              \end{array}
             \right.
      \end{equation}
     as $n\rightarrow\infty,$ where the last equality of $(\ref{50e})$ comes from $[\widehat{Z}_{n,1}|0]=I_{n},$ $I_{n}$ is
     the identity matrix of order $n.$\\

     For $g>1,$ the application of the H\"{o}lder inequality yields $\|T_{n}(h)[\widehat{Z}_{n,g}|0]\|_{1}\leq\|T_{n}(h)\|
     \|[\widehat{Z}_{n,g}|0]\|_{1}.$ Furthermore, it is easy to prove that
     $[\widehat{Z}_{n,g}|0]^{'}[\widehat{Z}_{n,g}|0]=\begin{bmatrix}
                                                                                                                I_{\mu_{g}} &0 \\
                                                                                                                0 & 0 \\
                                                                                                              \end{bmatrix},$
    where $'$ denotes the conjugate transpose, $I_{\mu_{g}}$ is the identity matrix of size $\mu_{g}=
    \left\lceil\frac{n}{g}\right\rceil$ and $\left\lceil\frac{n}{g}\right\rceil$ is the smallest integer greater than
    $\frac{n}{g}.$ Using the relation $[\widehat{Z}_{n,g}|0]^{'}[\widehat{Z}_{n,g}|0]=\begin{bmatrix}
                                                                                                     I_{\mu_{g}} &0 \\
                                                                                                     0 & 0 \\
                                                                                                     \end{bmatrix},$ it holds $\|[\widehat{Z}_{n,g}|0]\|_{1}
                                                                                                     =\left\lceil\frac{n}{g}\right\rceil.$
    This equality together with relation $(\ref{50e})$ provide
    \begin{equation}\label{51e}
     \|T_{n,g}(h))-T_{n}(\theta_{g}h)\|_{1}\leq \|h\|\left\lceil\frac{n}{g}\right\rceil+o(n),\text{\,\,\,\,}n\rightarrow\infty.
    \end{equation}
     Using estimate $(\ref{51e}),$ there exists a square matrix $N_{n}\in\mathbb{C}^{n\times n},$ with $\|N_{n}\|_{1}=
    \|h\|\cdot\left\lceil\frac{n}{g}\right\rceil$ and $\|N_{n}\|\leq\widehat{C},$ where $\widehat{C}$ is a positive constant
    independent of $n$ (for example, take $N_{n}=\|h\|\cdot[\widehat{Z}_{n,g}|0]$) such that,
     \begin{equation}\label{52e}
     \|T_{n,g}(h))-T_{n}(\theta_{g}h)-N_{n}\|_{1}\leq o(n),\text{\,\,\,\,}n\rightarrow\infty.
    \end{equation}
     Utilizing $(\ref{52e}),$ straightforward computations show that the sequences of matrices $\{T_{n,g}(h)\}$ and $\{T_{n}(\theta_{g}h)+N_{n}\}$ are asymptotically equivalent. Let $P$ be any fixed polynomial of
     degree $m,$ independent of $n.$ Applying Lemma $\ref{31l},$ $\{P(T_{n,g}(h))\}$ and $\{P(T_{n}(\theta_{g}h)+N_{n})\}$
     are asymptotically equivalent. This implies
     \begin{equation}\label{53e}
     \|P(T_{n,g}(h))-P(T_{n}(\theta_{g}h)+N_{n})\|_{1}= o(n),\text{\,\,\,\,}n\rightarrow\infty.
     \end{equation}
      Combining triangular inequality and relation $(\ref{53e})$ provides
      \begin{equation}\label{54e}
     \|P(T_{n,g}(h))\|_{1}\leq\|P(T_{n}(\theta_{g}h)+N_{n})\|_{1}+o(n),\text{\,\,\,\,}n\rightarrow\infty.
     \end{equation}
      Furthermore, since $T_{n}(\theta_{g}h),$ $N_{n}$ are uniformly bounded by a positive constant $\widehat{C}_{2},$ independent of
     $n$ and $P$ is a fixed polynomial of degree $m,$ a combination of estimate $(\ref{54e})$ and Lemma $\ref{48l}$ shows that
     there is a positive constant $\widehat{C}_{m}$ that depends on both polynomial $P$ and parameter $\widehat{C}_{2},$ but
     independent on $n$ such that,
      \begin{equation}\label{55e}
     \|P(T_{n,g}(h))\|_{1}\leq\|P(T_{n}(\theta_{g}h)\|_{1}+\widehat{C}_{m}\cdot\|h\|\cdot\left\lceil\frac{n}{g}\right\rceil+o(n),
     \text{\,\,\,\,}n\rightarrow\infty.
     \end{equation}
      Finally, taking into consideration Lemmas $\ref{16l}$ and $\ref{33l}$ together with Proposition $\ref{29p},$ we see
     that the matrix sequences $\{A_{n,g}(f_{1},f_{2})\}$ and $\{T_{n,g}(h)\}$ are asymptotically equivalent (in the
     sense of Definition $\ref{27d}$).  Applying again Lemma $\ref{31l},$ we obtain
     \begin{equation}\label{59}
     \|P(A_{n,g}(h))-P(T_{n,g}(h))\|_{1}= o(n),\text{\,\,\,\,} n\rightarrow\infty,
    \end{equation}
     for any fixed polynomial $P.$ Now, combining the triangular inequality and estimate $(\ref{59})$ yields
     \begin{equation*}
     \|P(A_{n,g}(f_{1},f_{2}))\|_{1}\leq\|P(T_{n,g}(h))\|_{1}+o(n),\text{\,\,\,\,} n\rightarrow\infty.
     \end{equation*}
     This fact together with estimates $(\ref{55e})$ and $(\ref{50})$ give
     \begin{equation}\label{59a}
      \|P(A_{n,g}(f_{1},f_{2}))\|_{1}\leq\frac{n}{mes(G)}\int_{G}\left|P((\theta_{g}h)(t))\right|dt
      +\widehat{C}_{m}\cdot\|h\|\cdot\left\lceil\frac{n}{g}\right\rceil+o(n),\text{\,\,\,\,} n\rightarrow\infty.
     \end{equation}
     The proof of Proposition $\ref{30p}$ is ended thanks to estimate $\left\lceil\frac{n}{g}\right\rceil
     \leq\frac{n}{g}+1$ and the limit: $\underset{n\rightarrow\infty}{\lim}\frac{o(n)}{n}=0.$
      \end{proof}
     Therefore, requirement \textbf{(i3)} of Lemma $\ref{20l}$ is established.\\

     We are now ready to prove Theorem $\ref{23t}.$

     \begin{proof} (of Theorem $\ref{23t}$)
     Suppose that $f_{1},f_{2}\in L^{\infty}(G),$ and set $h=f_{1}f_{2},$ $A_{n,g}(f_{1},f_{2})=T_{n,g}(f_{1})T_{n,g}(f_{2}).$
     Using Lemma $\ref{33l},$ the matrix sequence $\{A_{n,g}(f_{1},f_{2})\}$ is uniformly bounded by a positive constant independent of $n.$
     So item \textbf{(i1)} of Lemma $\ref{20l}$ is satisfied. Now, since $h\in L^{\infty}(G),$ is real-valued, $h$ satisfies the assumptions
     of Theorem $\ref{21t},$ then
     $\{T_{n,g}(h)\}\sim_{\lambda}(\theta_{g}h,G),$ and according to the definition of $\theta^{(g)}_{h}$ (see Theorem $3.1$ in
     \cite{en}) and the formula of $\theta_{g}$ (given by relation $(\ref{22})$), it easy to see that $\theta^{(g)}_{h}=\theta_{g}h.$
     Using the estimate $|tr(X)|\leq\|X\|_{1}$ for a square matrix $X$ (for example, see \cite{2sss}, Theorem $II.3.6$,
      Eq. $(II.23)$), for every non negative integer $d,$ we have
     \begin{equation}\label{60}
      \left|tr\left((A_{n,g}(f_{1},f_{2}))^{d}-(T_{n,g}(h))^{d}\right)\right|\leq\left\|(A_{n,g}(f_{1},f_{2}))^{d}-
     (T_{n,g}(h))^{d}\right\|_{1}.
     \end{equation}
     The aim of estimate $(\ref{60})$ is to prove that $\frac{1}{n}tr\left((A_{n,g}(f_{1},f_{2}))^{d}-(T_{n,g}(h))^{d}\right)$
     tends to zero when $n$ goes to infinity. Putting $P(z)=z^{d},$ an application of relation $(\ref{59})$ gives
      \begin{equation}\label{61}
       \left\|(A_{n,g}(f_{1},f_{2}))^{d}-(T_{n,g}(h))^{d}\right\|_{1}=o(n),\text{\,\,\,\,}n\rightarrow\infty.
      \end{equation}
     Using the property that the trace of a sum of matrices equals the sum of traces together with a combination of
     relations $(\ref{60})$ and $(\ref{61}),$ we get
     \begin{equation*}
     \underset{n\rightarrow\infty}{\lim}\left|\frac{1}{n}tr\left((A_{n,g}(f_{1},f_{2}))^{d}-(T_{n,g}(h))^{d}\right)\right|=0,
     \end{equation*}
     which implies
     \begin{equation}\label{62}
      \underset{n\rightarrow\infty}{\lim}\frac{1}{n}tr\left((A_{n,g}(f_{1},f_{2}))^{d}\right)=
      \underset{n\rightarrow\infty}{\lim}\frac{1}{n}tr\left((T_{n,g}(h))^{d}\right).
     \end{equation}
     Applying Theorem $\ref{21t},$ we have $\{T_{n,g}(h)\}\sim_{\lambda}(\theta_{g}h,G).$ Taking a function
     $F\in\mathcal{C}_{0}(\mathbb{C})$ defined by $F(z)=z^{d}$ $\forall z\in G,$ considering relation $(\ref{2})$ together with
     the property that the trace of a square matrix equals the sum of its eigenvalues, we have
     \begin{eqnarray}\label{63}
     % \nonumber to remove numbering (before each equation)
       \notag \underset{n\rightarrow\infty}{\lim}\frac{1}{n}tr\left((T_{n,g}(h))^{d}\right) &=&
       \underset{n\rightarrow\infty}{\lim}\frac{1}{n}\underset{j=1}{\overset{n}\sum}\lambda_{j}\left((T_{n,g}(h))^{d}\right)
       = \underset{n\rightarrow\infty}{\lim}\frac{1}{n}\underset{j=1}{\overset{n}\sum}F(\lambda_{j}(T_{n,g}(h)))= \\
        && \frac{1}{mes(G)}\int_{G}F((\theta_{g}h)(z))dz = \frac{1}{mes(G)}\int_{G}(\theta_{g}h)^{^{d}}(z)dz .
     \end{eqnarray}
     Combining relations $(\ref{62})$ and $(\ref{63}),$ requirement \textbf{(i2)} of Lemma $\ref{20l}$ is then satisfied. Now,
     restriction \textbf{(i3)} of Lemma $\ref{20l}$ is also satisfied according to Proposition $\ref{30p}.$
      An assembling of items \textbf{(i1)}, \textbf{(i2)} and \textbf{(i3)} (according to Lemma $\ref{20l}$) the sequence of matrices
     $\{A_{n,g}(f_{1},f_{2})\},$ is weakly clustered at $\mathcal{A}rea(\mathcal{ER}(\theta_{g}h)),$ and relation $(\ref{2})$
     holds for every function $F\in\mathcal{C}_{0}(\mathbb{C}),$ which is holomorphic in the interior of $\mathcal{A}rea(\mathcal{ER}
     (\theta_{g}h)).$ Finally, the function $\theta_{g}h\in L^{\infty}(G),$ belongs to the Tilli class. This means that
     $\mathcal{ER}(\theta_{g}h)$ does not disconnect the complex field and has an empty interior, the last
      condition \textbf{(i4)} of Lemma $\ref{20l}$ is also satisfied. Applying Lemma $\ref{20l},$ we obtain
     \begin{equation}\label{64}
      \{T_{n,g}(f_{1})T_{n,g}(f_{1})\}\sim_{\lambda}(\theta_{g}f_{1}f_{2},G),
     \end{equation}
     where the function $\theta_{g}$ is given by relation $(\ref{22}).$
     \end{proof}

   \begin{remark}
   Although the main result of this paper considered the additional assumption that the arithmetic averages of Fourier sums
   of order $r,$ with $r\leq m,$ (that is, the polynomials $P_{m,f_{1}}$ and $P_{m,f_{2}}$) of $f_{1}$ and $f_{2},$ respectively,
   satisfy the condition: for every $l\in\left\{\left[\frac{m}{g}\right],\left[\frac{m+1}{g}\right],..., \left[\frac{n-1-m}{g}
    \right]\right\},$ $T_{n,g}^{P_{m,f_{1}}P_{m,f_{2}}}\left(\exp(iglt)-\exp(ilt)\right)=0.$ This hypothesis played an important
    role in the proof of Proposition $\ref{28p},$ which was crucial in proving both Proposition $\ref{29p}$ and Theorem $\ref{23t}$.
    However, for $g\geq2$ the purpose of the work is to show that the product of $g$-Toeplitz sequences is clustered at zero,
    which is important in the context of the preconditioning problem.
   \end{remark}

   \section{Generalization to block and multilevel setting}\label{IV}
   We start this section by recalling that it is proven in \cite{sss} that the sequence $\{T_{n}(f_{1})T_{n}(f_{2})\}$
   is distributed (in the sense of eigenvalues) as the symbol $h=f_{1}f_{2}$ if $f_{1},f_{2}\in L^{\infty}(\mathbb{T}^{d})$ and $h$
   is real-valued ($d\in\mathbb{N}$, $d\geq1$, $\mathbb{T}=(-\pi,\pi)$). Furthermore, $\mathcal{ER}(h)$ is a weak cluster for
   $\{T_{n}(f)T_{n}(f_{2})\}$ and any $s\in \mathcal{ER}(h)$ strongly attracts the spectra of $\{T_{n}(f_{1})T_{n}(f_{2})\}$ with infinite order.
   This fact is sufficient for extending the proof of the relation $\{T_{n,g}(f_{1})T_{n,g}(f_{2})\}\sim_{\lambda}(\theta_{g}f_{1}f_{2},\mathbb{T})$
    to the case where $\theta_{g}$ is defined as in $(\ref{22})$ and $f_{1},f_{2}\in L^{\infty}(\mathbb{T}^{d}).$\\

    Let us consider the general multilevel case, where $f_{1},f_{2}\in L^{\infty}(\mathbb{T}^{d})$ are chosen to be matrix-valued.
    When $g$ is a positive vector, we have
    \begin{equation}
        \{T_{n,g}(f_{1})T_{n,g}(f_{2})\}\sim_{\lambda}(\theta_{g}f_{1}f_{2},\mathbb{T}^{d})
    \end{equation}
    where
    \begin{equation}
    \theta_{g}f_{1}f_{2}=\left\{
                       \begin{array}{ll}
                         f_{1}f_{2} & \hbox{if $g=e$}; \\
                         0 & \hbox{ for $g>e$.}
                       \end{array}
                     \right.,\text{\,\,\,\,\,\,\,\,\,\,\,\,\,\,\,\,}\mathcal{ER}(\theta_{g}f_{1}f_{2})=\left\{
                          \begin{array}{ll}
                            \mathcal{ER}(f_{1}f_{2}) & \hbox{if $g=e$}; \\
                            \{\underline{0}\} & \hbox{for $g>e$.}
                          \end{array}
                        \right.
   \end{equation}
   All the arguments are extended componentwise, that is, $g=e$ and $g>e$, respectively, means that $g_{r}=1$
   and $g_{r}>1,$ for $r=1,...,d$. In addition, $\mathcal{ER}(\theta_{g}f_{1}f_{2}),$ is a weak cluster for $\{T_{n,g}(f_{1})
   T_{n,g}(f_{2})\}$ (in the sense of Definition $\ref{6d}$) and any $s\in \mathcal{ER}(\theta_{g}f_{1}f_{2}),$
   strongly attracts the spectra of $\{T_{n,g}(f_{1})T_{n,g}(f_{2})\},$ with an infinite order.\\

   In the following we present some numerical experiments which confirm the theoretic analysis.

   \section{Some numerical examples}\label{V}
   This section deals with a wide set of numerical experiments which confirms our theoretical results. We analyze in detail
   different situations of the eigenvalue distribution which cover the theoretic analysis in the cases where the symbols
   $f_{1},f_{2}\in L^{\infty}(\mathbb{T})$ satisfy some restrictions: $(e1)$ $f_{1},\text{\,}f_{2}$ are polynomials
   (for instance, $f_{1}(x)=1+x+ix^{2}$ and $f_{2}(x)=1+4x^{3}-ix^{2}$); $(e2)$  $f_{1},\text{\,}f_{2}$ are rational functions
   (for example, $f_{1}(x)=\frac{x}{1+x^{2}}+i\frac{1-x}{1+2x^{2}}$ and $f_{2}(x)=\frac{2+x}{3+x^{2}}+i\frac{x^{2}}{1+x^{2}}$);
   $(e_{3})$ $f_{1},\text{\,}f_{2}$ are trigonometric polynomials (for instance, $f_{1}(x)=\exp(ix),$ $f_{2}(x)=3\exp(i2x)$) and
   $(e4)$  $f_{1},$ $f_{2}$ are two-variable functions (for instance, $f_{1}(x,y)=3+x+iy^{2}$ and
   $f_{2}(x,y)=\frac{y}{1+x^{2}}+i\frac{1-x}{1+y^{2}}$). Each item deals with different values of positive integer $n$ and the parameter $g.$\\

    In these numerical experiments, we consider four test cases and we report for each considered case the eigenvalues of
    $T_{n,g}(f_{1})T_{n,g}(f_{2}).$ We construct tables of two rows. The first one denoted by $N_{n,\epsilon},$ shows the cardinality
    of the eigenvalues (those greater than some positive epsilon in absolute value) of $\{T_{n,g}(f_{1})T_{n,g}(f_{2})\},$
    while the second one designated by $r_{n,\epsilon}=n/N_{n,\epsilon},$ represents the rate. We observe from the
    tables related to tests $1,2,4$ that when the parameter $g$ is strictly greater than $1$, $N_{n,\epsilon}$ is more and more less than $\left\lceil\frac{n}{g}\right\rceil$. The numerical tests have been developed with MatLab $R2009a,$ and the eigenvalues have been
    computed by the built-in MatLab function eig().\\

    $\bullet$ \textbf{Test 1:} $g=2,\text{\,}5,\text{\,}10,\text{\,}20;$ $\epsilon=10^{-1},\text{\,}(2.10)^{-1},\text{\,}(5.10)^{-1},
    \text{\,}10^{-2};$ $f_{1}(x)=1+x+ix^{2},$ $f_{2}(x)=1+4x^{3}-ix^{2};$ $N_{n,\epsilon}=cardinality\{\lambda\in\Lambda(T_{n,g}(f_{1})T_{n,g}(f_{2})):\text{\,\,}|\lambda|\geq\epsilon\}$ and $r_{n,\epsilon}=N_{n,\epsilon}/n.$
        \begin{equation*}
     \begin{array}{cc}
        \begin{tabular}{|c|c|c|c|c|}
                  \hline
          $g$              & 2      & 2    & 2  & 2  \\
                \hline
          $n$              & 50      & 100 &  200 & 400 \\
            \hline
          $N_{n,10^{-1}}$ & 25      &  25   &  25  &  25 \\
            \hline
           $r_{n,10^{-1}}$ & 0.5000 & 0.2500 & 0.1250 & 0.0625 \\
            \hline
         \end{tabular}     &    \begin{tabular}{|c|c|c|c|c|}
                  \hline
          $g$              & 5     & 5    & 5  & 5 \\
                \hline
          $n$              & 50    & 100 &  200 &  400 \\
            \hline
          $N_{n,(2.10)^{-1}}$ & 2   & 4 &  6  &  11  \\
            \hline
           $r_{n,(2.10)^{-1}}$ & 0.0400 &0.0400 &0.0300 &0.0275 \\
            \hline
         \end{tabular} \\
         \text{\,}\\
            \begin{tabular}{|c|c|c|c|c|}
                  \hline
          $g$              & 10      & 10    & 10  & 10 \\
                \hline
          $n$              & 50     & 100 &  200 &  400 \\
            \hline
          $N_{n,(5.10)^{-1}}$& 1   &  1   &  1  &  2 \\
            \hline
           $r_{n,(5.10)^{-1}}$ & 0.0200& 0.0100& 0.0050& 0.0050\\
            \hline
         \end{tabular}  &    \begin{tabular}{|c|c|c|c|c|}
                  \hline
          $g$              & 20  & 20    & 20  & 20 \\
                \hline
          $n$              & 50 & 100    & 200 &  400 \\
            \hline
          $N_{n,10^{-2}}$ &  5  &  1     &  1  &  1 \\
            \hline
           $r_{n,10^{-2}}$&0.1000&0.0100 & 0.0050&0.00250 \\
            \hline
         \end{tabular}
     \end{array}
     \end{equation*}

     The tables suggest that the rate $"r_{n,\epsilon}"$ approaches zero when $n$ increases.\\

       $\bullet$ \textbf{Test 2:} $g=2,\text{\,}5,\text{\,}10,\text{\,}20;$ $\epsilon=10^{-1},\text{\,}(2.10)^{-1},\text{\,}(5.10)^{-1},
    \text{\,}10^{-2};$ $f_{1}(x)=\frac{x}{1+x^{2}}+i\frac{1-x}{1+2x^{2}},$ $f_{2}(x)=\frac{2+x}{3+x^{2}}+i\frac{x^{2}}{1+x^{2}};$
    $N_{n,\epsilon}=cardinality\{\lambda\in\Lambda(T_{n,g}(f_{1})T_{n,g}(f_{2})):\text{\,\,}|\lambda|\geq\epsilon\}$ and $r_{n,\epsilon}=N_{n,\epsilon}/n.$

        \begin{equation*}
     \begin{array}{cc}
         \begin{tabular}{|c|c|c|c|c|}
                  \hline
          $g$              & 2      & 2    & 2  & 2 \\
                \hline
          $n$              & 50     & 100 &  200 &  400 \\
            \hline
          $N_{n,10^{-1}}$ &  25    &  25   &  25   & 25  \\
            \hline
           $r_{n,10^{-1}}$ & 0.5000 & 0.2500 & 0.1250& 0.0625  \\
            \hline
         \end{tabular}     &    \begin{tabular}{|c|c|c|c|c|}
                  \hline
          $g$                 & 5       & 5    & 5    & 5 \\
                \hline
          $n$                 & 50      & 100  &  200 &  400 \\
            \hline
          $N_{n,(2.10)^{-1}}$ &  1      & 3    &  5   &  9 \\
            \hline
           $r_{n,(2.10)^{-1}}$ & 0.0200 &0.0300&0.0250&0.0225 \\
            \hline
         \end{tabular} \\
         \text{\,}\\
           \begin{tabular}{|c|c|c|c|c|}
                  \hline
          $g$              & 10      & 10    & 10  & 10 \\
                \hline
          $n$              & 50      & 100 &  200 &  400 \\
            \hline
          $N_{n,(5.10)^{-1}}$ & 1   &  1  &  1   &  1 \\
            \hline
           $r_{n,(5.10)^{-1}}$ & 0.0200& 0.0100&0.0050&0.0025  \\
            \hline
         \end{tabular}  &    \begin{tabular}{|c|c|c|c|c|}
                  \hline
          $g$           & 20      & 20    & 20  & 20 \\
                \hline
          $n$           & 50      & 100 &  200 &  400 \\
            \hline
          $N_{n,10^{-2}}$ &  1    &  1   &   1  &  1  \\
            \hline
           $r_{n,10^{-2}}$ & 0.0200& 0.0100& 0.0050& 0.0025\\
            \hline
         \end{tabular}
     \end{array}
     \end{equation*}

      The tables show that the rate $"r_{n,\epsilon}"$ quickly approaches zero when $n$ becomes large. So the $g$-Toeplitz sequences $\{T_{n,g}(f_{1})T_{n,g}(f_{2})\}$ are strongly clustered at zero.\\

     $\bullet$ \textbf{Test 3:} $g=2,\text{\,}5,\text{\,}10,\text{\,}20;$ $\epsilon=10^{-1},\text{\,}(2.10)^{-1},\text{\,}(5.10)^{-1},
    \text{\,}10^{-2};$ $f_{1}(x)=\exp(ix),$ $f_{2}(x)=3\exp(-i2x);$
    $N_{n,\epsilon}=cardinality\{\lambda\in\Lambda(T_{n,g}(f_{1})T_{n,g}(f_{2})):\text{\,\,}|\lambda|\geq\epsilon\}$ and $r_{n,\epsilon}=N_{n,\epsilon}/n.$
        \begin{equation*}
     \begin{array}{cc}
         \begin{tabular}{|c|c|c|c|c|}
                  \hline
          $g$              & 2  & 2    & 2  & 2 \\
                \hline
          $n$              & 50 & 100 &  200 &  400 \\
            \hline
          $N_{n,10^{-1}}$ &  0  &  0   &  0   &  0 \\
            \hline
           $r_{n,10^{-1}}$ & 0  &   0  &  0   &  0 \\
            \hline
         \end{tabular}     &    \begin{tabular}{|c|c|c|c|c|}
                  \hline
          $g$              & 5  & 5    & 5  & 5 \\
                \hline
          $n$              & 50 & 100 &  200 &  400 \\
            \hline
          $N_{n,(2.10)^{-1}}$ & 0  &  0   &  0   &  0 \\
            \hline
           $r_{n,(2.10)^{-1}}$ & 0 &  0   &  0   &  0 \\
            \hline
         \end{tabular} \\
         \text{\,}\\
           \begin{tabular}{|c|c|c|c|c|}
                  \hline
          $g$              & 10   & 10    & 10  & 10\\
                \hline
          $n$              & 50  & 100 &  200 &  400 \\
            \hline
          $N_{n,(5.10)^{-1}}$ & 0   &  0   &  0   & 0  \\
            \hline
           $r_{n,(5.10)^{-1}}$ & 0  &   0  &  0   &  0 \\
            \hline
         \end{tabular}  &   \begin{tabular}{|c|c|c|c|c|}
                  \hline
          $g$              & 20   & 20  & 20  & 20 \\
                \hline
          $n$              & 50  & 100 &  200&  400 \\
            \hline
          $N_{n,10^{-2}}$ &   0   &  0   &   0   &  0   \\
            \hline
           $r_{n,10^{-2}}$ &  0  &  0   &    0  &  0   \\
            \hline
         \end{tabular}
     \end{array}
     \end{equation*}
    The tables show that the rate $"r_{n,\epsilon}"$ equals zero for every value of $n.$ So the sequence $\{T_{n,g}(f_{1})T_{n,g}(f_{2})\}$
    is strongly clustered at zero in the sense of eigenvalues.\\

     $\bullet$ \textbf{Test 4 (bidimensional case):} Setting $n=(n_{1},n_{2})$ and $|n|=n_{1}n_{2};$ \text{\,}$g=(g_{1},g_{2})=(2,2),
     \text{\,}(5,5),\text{\,}(10,10),\\
     \text{\,}(20,20);$ $\epsilon=10^{-1},\text{\,}(2.10)^{-1}, \text{\,}(5.10)^{-1},\text{\,}10^{-2};$ $f_{1}(x,y)=3+x+iy^{2},$ $f_{2}(x,y)=\frac{y}{1+x^{2}}+i\frac{1-x}{1+y^{2}};$ \\ $N_{n,\epsilon}=cardinality\{\lambda\in\Lambda(T_{n,g}(f_{1})T_{n,g}(f_{2})):\text{\,\,}|\lambda|\geq\epsilon\}$ and $r_{n,\epsilon}=N_{n,\epsilon}/|n|.$

        \begin{equation*}
     \begin{array}{cc}
         \begin{tabular}{|c|c|c|c|}
                  \hline
          $g$              & (2,2) & (2,2)    & (2,2)  \\
                \hline
          $n$              & (50,50)  & (100,100) & (200,200)  \\
            \hline
          $N_{n,10^{-1}}$ & 625  &  625   &  625   \\
            \hline
           $r_{n,10^{-1}}$ & 0.2500   & 0.0625    & 0.0156    \\
            \hline
         \end{tabular}     &    \begin{tabular}{|c|c|c|c|}
                  \hline
          $g$              & (5,5)  & (5,5)  & (5,5) \\
                \hline
          $n$              & (50,50)& (100,100)& (200,200) \\
            \hline
          $N_{n,(2.10)^{-1}}$ & 1  &  8    & 23  \\
            \hline
           $r_{n,(2.10)^{-1}}$ &0.0004 &0.0008 & 5.75e-4\\
            \hline
         \end{tabular} \\
         \text{\,}\\
           \begin{tabular}{|c|c|c|c|}
                  \hline
          $g$              & (10,10) & (10,10)    & (10,10) \\
                \hline
          $n$              & (50,50) & (100,100) & (200,200) \\
            \hline
          $N_{n,(5.10)^{-1}}$ & 1  & 1   & 1   \\
            \hline
           $r_{n,(5.10)^{-1}}$ & 0.0004 & 0.0001 & 2.5e-5   \\
            \hline
         \end{tabular}  &   \begin{tabular}{|c|c|c|c|}
                  \hline
          $g$              & (20,20)  & (20,20)  & (20,20) \\
                \hline
          $n$              & (50,50) & (100,100) & (200,200)\\
            \hline
          $N_{n,10^{-2}}$ &  2   &  1   & 1 \\
            \hline
           $r_{n,10^{-2}}$ & 0.0016 &0.0001& 2.5e-5  \\
            \hline
         \end{tabular}
     \end{array}
     \end{equation*}
     The tables suggest that the rate $"r_{n,\epsilon}"$ quickly approaches zero when $|n|$ becomes large. So the sequence
     $\{T_{n,g}(f_{1})T_{n,g}(f_{2})\},$ is clustered at zero in the sense of eigenvalues.\\

    A combination of four Tests shows the crucial role played by the product $\theta_{g}f_{1}f_{2}$ in the characterizing of eigenvalue
     distribution of products of $g$-Toeplitz structures. The cardinality of the eigenvalues (those greater than $\epsilon$ in absolute value)
     of $T_{n,g}(f_{1})T_{n,g}(f_{2})$ together with the rates agree with the corresponding theoretical results. In addition,
    the tables indicate that the distribution result (in the sense of eigenvalues) is subtle. It is not unconditionally "non distributed"
    for any values of the generating functions $f_{1}$ and $f_{2}$ along with the parameter $g$ (see for example the tables of tests).
    Furthermore, the third test suggests that if the symbols $f_{1}$ and $f_{2}$ are trigonometric polynomials the rates $r_{n,\epsilon}$
    are closer and closer to zero for any values of $n$ and $g$ (with $g\geq2$). This means that the sequences $\{T_{n,g}(f_{1})T_{n,g}(f_{2})\},$
    is clustered at zero (in the sense of eigenvalues) for every values of $n$ and $g$ (with $g\geq2$). When comparing the tables, it is
    easy to see that the bidimensionsional case provides good clustering to zero. Finally, the tests suggest that the additional assumption:
    for every $l\in\left\{\left[\frac{m}{g}\right],\left[\frac{m+1}{g}\right],..., \left[\frac{n-1-m}{g}\right] \right\},$ $T_{n,g}^{P_{m,f_{1}}P_{m,f_{2}}}\left(\exp(iglt)-\exp(ilt)\right)=0$ is not required to get a distribution result in the sense of eigenvalues
    for products of $g$-Toeplitz sequences. More specifically, we think that the only hypothesis $f_{1},f_{2}\in L^{\infty}(G)$ is sufficient
    to obtain a distribution result. We recall that this assumption was only essential in the proof of Proposition $\ref{28p},$ which was crucial in
    proving both Proposition $\ref{29p}$ and Theorem $\ref{23t}.$

    \section{Conclusion and future works}\label{VI}
    This paper has studied in detail the eigenvalue distribution of products of $g$-Toeplitz sequences\\
    $\{T_{n,g}(f_{1})T_{n,g}(f_{2})\},$ in the case where the generating functions $f_{1},f_{2}\in L^{\infty}(-\pi,\pi).$
    The analysis has shown that if: $(R1)$ $f_{1},f_{2}\in L^{\infty}(G)$ and $(R2)$ the $g$-Toeplitz matrix
    $T_{n,g}(P_{m,f_{1}}P_{m,f_{2}})$ related to the $g$-Toeplitz operator $T_{n,g}^{P_{m,f_{1}}P_{m,f_{2}}}$ (where the polynomials
    $P_{m,f_{i}},$ $i=1,2,$ are the arithmetic averages of Fourier sums of order $r,$ with $r\leq m,$ of the functions $f_{i},$ respectively)
    satisfy an additional condition, then the distribution result: $\{T_{n,g}(f_{1})T_{n,g}(f_{2})\}\sim_{\lambda}(0,G)$ holds when the
    parameter $g$ is greater than $1.$ This theoretical result was clearly confirmed by some numerical experiments in both one and two dimensions.
    Furthermore, Test $4$ shows that the good clustering to zero is obtained in bidimensional setting. Finally, the generalization
    of this result to the blocks and multilevel setting amounting to choose the matrix-valued symbols was also presented. According to
    information provided by tables, the question is to know if something can be said if requirement $(R1)$ is deleted. More specifically, is it possible to establish a distribution result (in the sense of eigenvalues) for products of $g$-Toeplitz sequences, $\{T_{n,g}(f_{1})T_{n,g}(f_{2})\}$ (case where $g>1$), when either requirement $(R1)$ or restriction $(R2)$ is not satisfied?
    The latter problem will be subject of our future investigations.\\

    \textbf{Acknowledgment.} The author thanks the anonymous referees for detailed and valuable comments which helped to greatly improve
      the quality of this paper.


\begin{thebibliography}{99}

    \bibitem{1nss}
     A. Aric\"{o}, M. Donatelli, S. Serra Capizzano, "V-cycle optimal convergence for certain (multilevel)
     structured linear systems". SIAM J. Matrix Anal. Appl., $26(2004)$, pp. $186$-$214$.

    \bibitem{1sss}
    F. Avram, "On bilinear forms on Gaussian random variables and Toeplitz matrices". Probab. Theory Related Fields,
    $79(1988)$, pp. $37$-$45$.

    \bibitem{2sss}
    R. Bathia, "Matrix Analysis". Spring Verbag, New York, $(1997)$.

     \bibitem{3sss}
     R. Bhatia, "Fourier Series", AMS, Providence, $2004.$

     \bibitem{13dgmss}
    Di Benedetto F., Fiorentino G., Serra S. C.G., "Preconditioning for Toeplitz matrices. Computers $\&$ Mathematics with
    Applications" $25$ $(1993)$ pp. $35$-$45.$

     \bibitem{8sss}
    A. B\"{o}ttcher and B. Silbermann, "Introduction to Large Truncated Toeplitz Matrices", Springer-Verlag, New York, $1999.$

    \bibitem{9sss}
    L. Brown and P. Halmos, "Algebraic properties of Toeplitz operators", J. Reine Angew. Math., $123$ $(1964),$ pp. $89$-$102.$

    \bibitem{7dgmss}
    Chan R. H., "Toeplitz preconditioners for Toeplitz systems with nonnegative generating functions", IMA Journal of Numerical
    Analysis $11$ $(1991)$ pp. $333$-$345.$

    \bibitem{5nss}
     I. Daubechies, "Ten Lectures on wavelets". CBMS-NSF Regional Conference Series in Applied Mathematics $61$, SIAM,
     Philadelphia, $(1992)$.

     \bibitem{7nss}
     N. Dyn, D. Levin, "Subdivision schemes in geometric modelling". Acta Numerica, $11(2002)$, pp. $73$-$144$.

     \bibitem{enss}
     C. Estatico, E. Ngondiep, S. Serra-Capizzano and D. Sesana. "A note on the (regularizing) preconditioning of g-Toeplitz sequences
     via g-circulants", J. Comput. Appl. Math. $236(2012)$, $2090$-$2111$, $22$ pages .

     \bibitem{10nss}
     G. Fiorentino and S. Serra-Capizzano, Multigrid methods for symmetric positive definite block Toeplitz matrices with
     nonnegative generating functions, SIAM J. Sci. Comput., $17(1996),$ pp. $1068$-$1081.$

     \bibitem{13sss}
     R.M. Gray, Toeplitz and Circulant Matrices: a Review, Foundations and Trends in Comm. Inf. Theory, $2$-$3$ $(2006),$ pp. $1$-$93.$

     \bibitem{13nss}
      W. Hackbush, "Multigrid Methods and Applications", Springer-Verlag, Berlin, $1985.$

     \bibitem{16sss}
     S. Holmgren, S. Serra-Capizzano, and P. Sundqvist, "Can one hear the composition of a drum", Mediterr. J. Math.,
     $3$-$2$ $(2006),$ pp. $227$-$249.$

      \bibitem{16dgmss}
      Huckle T., Serra-Capizzano S., Tablino Possio C., "Preconditioning strategies for non-Hermitian Toeplitz linear systems",
      Numerical Linear Algebra with Applications $12$ $(2005)$ pp. $211$-$220.$

     \bibitem{14nss}
     A. B. J. Kuijilaars, S. Serra Capizzano, "Asymptotic zero distribution of orthogonal polynomials with discontinuously varying
     recurrence coefficients". J. Approx. Theory, $13(2001)$, pp. $142$-$155$.

     \bibitem{19sss}
     I. Louhichi, E. Strouse, and L. Zakariasy, "Products of Toeplitz Operators on the Bergman space", Integral Equations and
     Operator Theory, $54$ $(2006),$ pp. $525$-$539.$

      \bibitem{en}
     E. Ngondiep, "Distribution in the sense of eigenvalues of g-Toeplitz sequences: Clustering and attraction", Arab J. Math. Sci.
     $22(2016),$ $45$-$60,$ $16$ pages.

     \bibitem{en1}
     E. Ngondiep, "How to determine the eigenvalues of g-circulant matrices", Operators and Matrices, $12(3)$ $(2019)$ $797$-$822$, $26$ pages. 

     \bibitem{ns}
     E. Ngondiep and S. Serra Capizzano. "Approximation and spectral analysis for large structured linear systems", LAP LAMBERT Academic
     Publishing. ISBN-$13$: $978$-$3$-$8454$-$1547$-$5$; ISBN-$10$: $3845415479$; EAN: $9783845415475$;  $268$ pages, $(2011).$

     \bibitem{nss}
      E. Ngondiep, S. Serra Capizzano, D. Sesana, "Spectral features and asymptotic properties of $g$-circulant and $g$-Toeplitz sequences".
      SIAM J. Matrix Anal. Appl., $31$-$4(2010)$, pp. $1663$-$1687$, $25$ pages.
      
      \bibitem{n1ss}
     E. Ngondiep, S. Serra-Capizzano, D. Sesana. "Spectral features and asymptotic properties for alpha-circulants and alpha-Toeplitz sequences: theoretical results and examples", preprint available online from http://arXiv:0906.2104, 2009.

     \bibitem{20dgmss}
     Noutsos D., Serra-Capizzano S., Vassalos P, "Matrix algebra preconditioners for multilevel Toeplitz systems do not insure
     optimal convergence rate", Theoretical Computer Science $315$ $(2004)$ pp. $557$-$579.$

     \bibitem{17en}
      S. V. Parter, "On the distribution on the singular values of Toeplitz matrices". Linear Algebra Appl., $80(1986)$,
      pp. $115$-$130$.

      \bibitem{22sss}
      W. Rudin, "Real and Complex Analysis", McGraw-Hill, New York, $1974.$

      \bibitem{27sss}
      S. Serra-Capizzano, "Generalized Locally Toeplitz sequences: spectral analysis and applications to discretized Partial
      Differential Equations", Linear Algebra Appl., $366$-$1$ $(2003),$ pp. $371$-$402.$

      \bibitem{28sss}
      S. Serra-Capizzano, "The GLT class as a Generalized Fourier Analysis and applications". Linear Algebra Appl., $419$-$1$
      $(2006)$, pp. $180$-$233.$

      \bibitem{18en}
      S. Serra Capizzano, "Spectral and Computational analysis of block Toeplitz matrices with nonnegative definite
      generating functions". BIT, $39(1999)$, pp. $152$-$175$.

      \bibitem{25sss}
      S. Serra Capizzano, "Spectral behavior of matrix sequences and discretized boundary value problems". Linear Algebra
      Appl., $337(2001)$, pp. $37$-$78$.

      \bibitem{19nss}
      S. Serra Capizzano, "A note on antireflective boundary conditions and fast deblurring models". SIAM J. Sci. Comput.,
      $25(2003)$, pp. $1307$-$1325$.

      \bibitem{17nss}
      S. Serra Capizzano, "Convergence analysis of two-grid methods for elliptic Toeplitz and PDEs matrix-sequences". Numer.
      Math., $92(2002)$, pp. $433$-$465$.

      \bibitem{29sss}
      S. Serra Capizzano, D. Bertaccini and G. H. Golub, "How to deduce a proper eigenvalue cluster from a proper singular
      value cluster in the non normal case", SIAM J. Matrix Anal.Appl., $27$-$1$ $(2005),$ pp. $82$-$86.$

      \bibitem{ss}
      S. Serra Capizzano, D. Sesana, "A note on the eigenvalues of g-circulants (and of g-Toeplitz, g-Hankel matrices)".
      Calcolo $51$ $(2014)$, no. $4$, $639$-$659$.

      \bibitem{sss}
      S. Serra Capizzano, D. Sesana, E. Strouse, "The eigenvalue distribution of product of Toeplitz matrices: clustering and
      attraction". Linear Algebra and Appl., $39(2010)$.

      \bibitem{31sss}
      S. Serra Capizzano and P. Tilli, "On unitarily invariant norms of matrix valued linear positive operators", J. Inequalities
      Appl., $7$-$3$ $(2002),$ pp. $309$-$330.$

      \bibitem{32sss}
       Serra Capizzano S., Tilli P, "Extreme singular values and eigenvalues of non-Hermitian block Toeplitz matrices", Journal
      of Computational and Applied Mathematics $108$ $(1999)$ pp. $113$-$130.$

      \bibitem{33dgmss}
      Serra-Capizzano S., Tyrtyshnikov E., "Any circulant-like preconditioner for multilevel Toeplitz matrices is not superlinear",
      SIAM Journal on Matrix Analysis and Applications $21$ $(1999)$ pp. $431$-$439.$

      \bibitem{23nss}
      G. Strang, "Wavelets and dilation equations: a brief introduction". SIAM Rev., $31(1989)$, pp. $614$-$627$.

      \bibitem{36sss}
      P. Tilli, "Some results on complex Toeplitz eigenvalues", J. Math. Anal. Appl., $239$-$2$ $(1999)$, pp. $390$-$401.$

      \bibitem{33sss}
      P. Tilli, "Singular values and eigenvalues of non-Hermitian block Toeplitz matrices", Linear Algebra Appl., $272$ $(1998),$
      pp. $59$-$89.$

      \bibitem{35sss}
      P. Tilli, "A note on the spectral distribution of Toeplitz matrices". Linear Multilin. Algebra, $45(1998)$,
      pp. $147$-$159$.

      \bibitem{27nss}
      U. Trottenberg, C. W. Oosterlee, and A. "Sch\"{u}ller, "Multigrid, Academic Press, San Diego, $2001.$

      \bibitem{38sss}
      E. Tyrtyshnikov and N. Zamarashkin, "Spectra of multilevel Toeplitz matrices: advanced theory via simple matrix
      relationships". Linear Algebra Appl., $270(1998)$, pp. $15$-$27$.

      \bibitem{41sss}
      A. Zygmund, "Trigonometric Series", Cambridge University Press, Cambridge, $1959.$

     \end{thebibliography}
     \end{document}